\documentclass[12pt]{article}
\usepackage{amsmath}
\usepackage{amsthm}
\usepackage{amsfonts}
\usepackage{eucal}
\newtheorem{thm}{Theorem}
\newtheorem{lem}{Lemma}
\newtheorem{prop}{Proposition}
\newtheorem{cor}{Corollary}
\newtheorem{conj}{Conjecture}
\numberwithin{thm}{section}
\numberwithin{lem}{section}
\numberwithin{prop}{section}
\numberwithin{cor}{section}
\numberwithin{conj}{section}
\numberwithin{equation}{section}
\def\al{\alpha}
\def\be{\beta}
\def\ga{\gamma}
\def\de{\delta}
\def\la{\lambda}
\def\si{\sigma}
\def\th{\theta}

\def\zt{\zeta}
\def\ep{\epsilon}
\def\om{\omega}

\def\Ga{\Gamma}

\def\A{\mathcal A}
\def\P{\mathcal P}
\def\B{\mathcal B}
\def\Z{\mathcal Z}

\def\H{\mathfrak H}
\def\E{\mathfrak E}
\def\Q{\mathbb Q}
\def\R{\mathbb R}
\def\C{\mathbb C}
\def\T{\mathcal T}
\def\L{\mathcal L}

\def\<{\langle}
\def\>{\rangle}

\def\dim{\operatorname{dim}}
\def\card{\operatorname{card}}
\def\Sy{\operatorname{Sym}}
\def\QS{\operatorname{QSym}}
\def\coeff{\operatorname{coeff}}
\def\ev{\operatorname{ev}}
\def\op{\diamond}
\def\Re{\operatorname{Re}}
\newmuskip\pFqskip
\mathchardef\pFcomma=\mathcode`,
\newcommand*\pFq[5]{%
\begingroup
\begingroup\lccode`~=`,
\lowercase{\endgroup\def~}{\pFcomma\mkern\pFqskip}%
\mathcode`,=\string"8000
{}_{#1}F_{#2}\biggl[\genfrac..{0pt}{}{#3}{#4};#5\biggr]%
\endgroup
}
\begin{document}
\title{An Odd Variant of Multiple Zeta Values}
\author{Michael E. Hoffman\\
\small Dept. of Mathematics, U. S. Naval Academy\\[-0.8ex]
\small Annapolis, MD 21402 USA\\[-0.8ex]
\small\texttt{meh@usna.edu}}
\date{January 29, 2019
\small MR Classification:  Primary 11M32; Secondary 05E05, 11M35, 33C20}
\maketitle
\begin{abstract}
For positive integers $i_1,\dots,i_k$ with $i_1>1$, we define the 
multiple $t$-value $t(i_1,\dots,i_k)$ as the sum of those terms of 
the usual infinite series for the multiple zeta value $\zt(i_1,\dots,i_k)$ 
with odd denominators.
Multiple $t$-values can be written as rational linear combinations of 
the alternating or ``colored'' multiple zeta values.
Using known results for colored multiple zeta values, we
obtain tables of multiple $t$-values through weight 7,
suggesting some interesting conjectures, including one that the 
dimension of the rational vector space generated by weight-$n$ 
multiple $t$-values has dimension equal to the $n$th Fibonacci number.
Like the multiple zeta values, the multiple $t$-values can 
be multiplied according to the rules of the harmonic algebra.
Using this fact, we obtain explicit formulas for multiple 
$t$-values with repeated arguments analogous to those known 
for multiple zeta values.
We express the generating function of the height one multiple
$t$-values $t(n,1,\dots,1)$ in terms of a generalized hypergeometric function.
We also define alternating multiple $t$-values and prove some results
about them.
\end{abstract}
\section{Introduction}
In the past few decades the multiple zeta values
\[
\zt(i_1,\dots,i_k)=\sum_{n_1>n_2>\dots>n_k\ge 1}
\frac1{n_1^{i_1}n_2^{i_2}\cdots n_k^{i_k}} 
\]
have appeared prominently in both number theory and physics.
In this paper we consider the related quantities
\begin{equation}
\label{mtv}
t(i_1,\dots,i_k)=\sum_{\substack{n_1>n_2>\dots>i_k\ge 1\\ \text{$n_i$ odd}}}
\frac1{n_1^{i_1}n_2^{i_2}\cdots n_k^{i_k}} ,
\end{equation}
which we call multiple $t$-values.  In both these definitions
$i_1,i_2,\dots, i_k$ are positive integers
with $i_1>1$; we call $k$ the ``depth'' and $i_1+\dots+i_k$ the
``weight.''
Our study reveals that multiple $t$-values have remarkable parallels
to, and contrasts with, multiple zeta values.
\par
Our notation is adapted from N. Nielsen \cite{N}, who wrote $t_n$
for $t(n)$ and gave the formula
\begin{equation}
\label{tniel}
\sum_{i=1}^{n-1}t(2i)t(2n-2i)=\frac{2n-1}2 t(2n) .
\end{equation}
This may be compared with the formula (also given in \cite{N})
\begin{equation}
\label{zniel}
\sum_{i=1}^{n-1}\zt(2i)\zt(2n-2i)=\frac{2n+1}2\zt(2n).
\end{equation}
Of course
\[
\zt(i)=\sum_{n\ge 1}\frac1{n^i}=\sum_{\text{$n$ odd}}\frac1{n^i}+
\sum_{\text{$n$ even}}\frac1{n^i}=t(i)+\frac1{2^i}\zt(i),
\]
so that $t(i)=(1-2^{-i})\zt(i)$ and the classical formula
\[
\zt(2n)=\frac{(-1)^{n-1}B_{2n}(2\pi)^{2n}}{2(2n)!}
\]
is paralleled by 
\begin{equation}
\label{teur}
t(2n)=\frac{(-1)^{n-1}B_{2n}(2^{2n}-1)\pi^{2n}}{2(2n)!} .
\end{equation}
Eq. (\ref{teur}) can be expressed by the generating function
\begin{equation}
\label{tangen}
\sum_{n=1}^\infty t(2n)x^{2n-1}=\frac{\pi x}4\tan\left(\frac{\pi x}2\right) ,
\end{equation}
from which Nielsen's formula (\ref{tniel}) follows easily by
differentiating both sides of Eq. (\ref{tangen}) and then 
comparing that to the result of squaring both sides.
\par
The multiplication of multiple $t$-values as series works just like
the multiplication of multiple zeta values as series, so that
we have, for example,
\begin{multline*}
t(2)t(3,1,1)=t(2,3,1,1)+t(5,1,1)+t(3,2,1,1)+t(3,3,1)+t(3,1,2,1)\\
+t(3,1,2)+t(3,1,1,2).
\end{multline*}
Consequently (paralleling \cite[Thm. 2.2]{H1992}), any symmetric sum of
multiple $t$-values, e.g., $t(3,2,2)+t(2,3,2)+t(2,2,3)$, is a rational
polynomial in ordinary $t$-values, e.g.,
\[
t(3,2,2)+t(2,3,2)+t(2,2,3)=\frac12t(2)^2t(3)-\frac12t(3)t(4)
-t(2)t(5)+t(7) .
\]
In particular, any multiple $t$-value with all its arguments equal to
the same integer $k$ is a rational polynomial in the $t$-values
$t(k),t(2k),t(3k),\dots$.
In view of Eq. (\ref{teur}) above, this means that, when $k$
is even, a multiple $t$-value of the form $t(k,k,\dots,k)$ (with $n$ 
repetitions) is a rational multiple of $\pi^{nk}$, just as with multiple zeta
values.
For example, the well-known identities
\[
\zt(\{2\}_n)=\frac{\pi^{2n}}{(2n+1)!} ,\quad
\zt(\{4\}_n)=\frac{2^{2n+1}\pi^{4n}}{(4n+2)!} ,\quad
\zt(\{6\}_n)=\frac{6(2\pi)^{6n}}{(6n+3)!}
\]
(where $\{k\}_n$ means $k$ repeated $n$ times) have multiple
$t$-value counterparts
\begin{equation}
\label{krep}
t(\{2\}_n)=\frac{\pi^{2n}}{2^{2n}(2n)!} ,\quad
t(\{4\}_n)=\frac{\pi^{4n}}{2^{2n}(4n)!} ,\quad
t(\{6\}_n)=\frac{3\pi^{6n}}{4(6n)!} .
\end{equation}
\par
As with multiple zeta values, one can define $t$-star values by
\begin{equation}
\label{mtsv}
t^{\star}(i_1,\dots,i_k)=\sum_{\substack{n_1\ge n_2\ge\cdots\ge n_k\ge 1\\ \text{$n_i$ odd}}}
\frac1{n_1^{i_1}n_2^{i_2}\cdots n_k^{i_k}} 
\end{equation}
for positive integers $i_1,\dots,i_k$ with $i_1>1$.
S. Muneta \cite{M} gave an identity expressing $\zt^{\star}(\{2m\}_n)$ as 
$\pi^{2mn}$ times a rational polynomial in Bernoulli numbers; similarly, 
in Theorem \ref{eulern} below we give an identity expressing 
$t^{\star}(\{2m\}_n)$ as $\pi^{2mn}$ times a rational polynomial in 
Euler numbers.  
The case $m=1$ is
\[
t^{\star}(\{2\}_n)=\frac{\pi^{2n}}{2^{2n}(2n)!}(-1)^nE_{2n} .
\]
\par
Despite the parallels, the algebra of multiple $t$-values is
quite different in some ways from the algebra of multiple zeta values.
Both the duality theorem and the ``double shuffle relations'' \cite{IKZ}
of multiple zeta values are missing for multiple $t$-values.
The difference already appears in weight 3:  while for multiple
zeta values there is the famous identity $\zt(2,1)=\zt(3)$,
for multiple $t$-values one has instead
\begin{equation}
\label{t21}
t(2,1)=-\frac12t(3)+t(2)\log 2 .
\end{equation}
Nevertheless, as we show in \S4, any multiple $t$-value is a
rational linear combination of ordinary and alternating (or ``colored'')
multiple zeta values, the latter being forms such as
\[
\zt(\bar 3,1)=\sum_{i>j\ge 1}\frac{(-1)^i}{i^3j} .
\]
Existing tables of such values \cite{BJOP,BBV} can be used to give 
formulas for multiple $t$-values in terms of alternating multiple
zeta values; in Appendix A we give such formulas through weight 7.
As with multiple zeta values, all known relations of multiple $t$-values
are homogeneous by weight.
\par
Examination of the tables leads to several conjectures, which are  
presented in \S2 below.
Conjecture \ref{der} asserts that the algebra of multiple
$t$-values admits a weight-decreasing derivation, which can easily
be seen not to exist in the case of multiple zeta values.
Conjecture \ref{fib} states that the rational vector space generated
by the multiple $t$-values of weight $n$ has dimension
equal to the $n$th Fibonacci number.
(Conjecture \ref{sah}, due to B. Saha \cite{Sa}, gives this a concrete
form by proposing a basis for the weight-$n$ multiple $t$-values.)
This compares to the well-known conjecture that the dimension of
the rational vector space of weight-$n$ multiple zeta values is the
$n$th Padovan number.
\par
In \S3 we prove the analogue for multiple $t$-values of the symmetric
sum theorem for multiple zeta values \cite{H1992}.  This implies the
the results for repeated arguments given in Eqs. (\ref{krep}) above.
In \S4 we show how any multiple $t$-value can be written as a sum
of alternating multiple zeta values.
\par
The ``height one'' multiple zeta values $\zt(n,1,\dots,1)$ are rational 
polynomials in the ordinary zeta values $\zt(i)$, $i\ge 2$; indeed this
follows from the generating-function identity
\cite{BBB,H1997}
\begin{equation}
\label{mzvht1}
\sum_{i,j\ge 1}\zt(i+1,\{1\}_{j-1})x^iy^j=
1-\frac{\Ga(1-x)\Ga(1-y)}{\Ga(1-x-y)} .
\end{equation}
It is already apparent from
\[
t(3,1)=-\frac{37}{60}t(4)-\frac12\zt(\bar3,1)+t(3)\log2
\]
that multiple $t$-values of height one are more complicated.  
Nevertheless, in \S5 we express
\[
H(x,y)=\sum_{i,j\ge 1}t(i+1,\{1\}_{j-1})x^iy^j
\]
as the value of a generalized hypergeometric function (Theorem 
\ref{hypg32} below).
In contrast with the generating function (\ref{mzvht1}) for 
height one multiple zeta values, which is symmetric in $x$ and $y$, 
$H(x,y)$ is very far from symmetric; e.g., the series
\[
t(2)+t(2,1)+t(2,1,1)+t(2,1,1,1)+\cdots
\]
converges (and in fact converges to twice Catalan's constant; see 
Eq. (\ref{cat}) below), while
\[
t(2)+t(3)+t(3)+t(4)+\cdots ,
\]
like the corresponding series of multiple zeta values, diverges.
\par
In \S6 we define alternating multiple $t$-values in a manner
analogous to alternating multiple zeta values.  Then one can
write a formula for $t(\bar n,\dots,\bar n)$ in terms of the
$t$-values of even integers and the values of the Dirichlet beta
function at odd integers, and in fact there are explicit formulas
\[
t(\{\bar1\}_k)=(-1)^{\lfloor\frac{k+1}2\rfloor}\frac{\pi^k}{2^{2k}k!}\quad
\text{and}\quad
t(\{\bar 3\}_k)=(-1)^{\lfloor\frac{k+1}2\rfloor}\frac{3\pi^{3k}}{2^{3k+1}(3k)!}.
\]
\par
The multiple $t$-value (\ref{mtv}) can be written in terms of the 
Hurwitz multiple zeta function
\[
\zt(i_1,\dots,i_k;a_1,\dots,a_k)=
\sum_{n_1>\dots>n_k\ge 1}\frac1{(n_1+a_1)^{i_1}\cdots (n_k+a_k)^{i_k}}
\]
discussed in \cite{MS}; taking $a_1=\dots=a_k=-\frac12$, we have
\begin{equation}
\label{hmz}
t(i_1,\dots,i_k)=2^{i_1+\dots+i_k}\zt(i_1,\dots,i_k;-\tfrac12,\dots,-\tfrac12) .
\end{equation}
Double $t$-values $t(n,m)$ are referred to as ``double zeta values of
level 2'' in \cite{KT} and \cite{NT}, where they are written as
$\zt^{\mathbf{oo}}(m,n)$ and $\zt^{\mathbf{o}}(m,n)$ respectively.
\section{Conjectures on the algebra of multiple $t$-values}
\par
Let $\H^1$ be the underlying rational vector space of the noncommutative 
polynomial algebra $\Q\<z_1,z_2,\dots\>$, and let $\H^0$ be the 
subalgebra of of $\H^1$ generated by 1 and those monomials that start
with $z_i$, $i>1$.
From \cite{H1997} we have the following result.
\begin{thm}
\label{quas}
The rational vector space $\H^1$ with the product $*$ defined recursively
by $w_1*1=1*w_1=w_1$ and
\begin{equation}
\label{recur}
z_iw_1*z_jw_2=z_i(w_1*z_jw_2)+z_j(z_iw_1*w_2)+z_{i+j}(w_1*w_2)
\end{equation}
for all monomials $w_1,w_2$ of $\H^1$ is a commutative algebra.  Further, 
$\H^1$ is a polynomial algebra, $\H^0$ is a subalgebra, and $\H^1=\H^0[z_1]$.
\end{thm}
\par
We recall that the quasi-symmetric functions $\QS$ are the set
of formal power series $f$ in $x_1,x_2,\dots$ of bounded degree such that,
for any sequence $i_1<i_2<\dots<i_p$, the coefficient of
\[
x_{i_1}^{n_1}x_{i_2}^{n_2}\cdots x_{i_p}^{n_p}
\]
in $f$ is the same as the coefficient in $f$ of
\[
x_1^{n_1}x_2^{n_2}\cdots x_p^{n_p} .
\]
Then $\QS$ is an algebra, and it contains the symmetric functions
$\Sy$ as a proper subalgebra.  
As a vector space, $\QS$ is generated by the monomial quasi-symmetric
functions
\[
M_{n_1,\dots,n_p}=\sum_{i_1<\dots<i_p}x_{i_1}^{n_1}\cdots x_{i_p}^{n_p} .
\]
In \cite{H1997} the following result is proven.
\begin{thm}
\label{sym}
$(\H^1,*)$ is isomorphic to the algebra $\QS$ of quasi-symmetric functions
via the map that sends $z_{n_1}\cdots z_{n_p}$ to $M_{n_1,\dots,n_p}$.
\end{thm}
Henceforth we will identify $\H^0$ with the subalgebra $\QS^0$ of $\QS$
generated by $M_{n_1,\dots,n_p}$ with $n_1>1$.  Let $\T$ be the subspace
of $\R$ generated over the rationals by 1 and the multiple $t$-values.
Theorem \ref{tquas} below gives a homomorphism $\th:\H^0\to\T$, which
we can also regard as a homomorphism from $\QS^0$ to $\T$.
Now in \cite{H2009} it is shown that the linear map $A_-:\QS\to\QS$
with $A_-(1)=0$ and
\[
A_-(M_{(i_1,\dots,i_k)})=\begin{cases} M_{(i_1,\dots,i_{k-1})},&\text{if $i_k=1$,}\\
0,&\text{otherwise,}\end{cases}
\]
is a derivation.  It is also evident that $A_-(\QS^0)\subset\QS^0$.
We make the following conjecture.
\begin{conj} 
\label{der}
The algebra $\T$ admits a derivation $d$ such that $d\th=\th A_-$.
\end{conj}
By Corollary \ref{expr} below, $2^k$ times any multiple $t$-value of
depth $k$ is a signed sum of alternating multiple zeta values.
Now the Multiple Zeta Value Data Mine \cite{BBV} gives formulas for
all alternating multiple zeta values through weight 12 as rational
linear combinations of what are believed to be basis elements (all the
relations there are proved, but it is possible that the ``basis'' is
really just a spanning set, as undiscovered rational relations may exist).
Using this resource, we have expressed multiple $t$-values through
weight 7 as rational linear combinations of products of (1) ordinary
$t$-values $t(n)$, $n\ge 2$;
(2) $\log2$; and (3) selected alternating multiple zeta values.
The particular alternating multiple zeta values used are taken
from the conjectural basis used in \cite{BBV}; through weight 7 they are
\[
\zt(\bar3,1),\ \zt(\bar3,1,1),\ \zt(\bar5,1),\ \zt(\bar3,1,1,1),\
\zt(\bar5,1,1),\ \zt(\bar3,3,1),\ \zt(\bar3,1,1,1,1) .
\]
These formulas are listed in the Appendix A below.  For these 
formulas, $d$ is formal differentiation with respect to $\log 2$.
For example,
\begin{multline*}
t(2,3,1)=-\frac2{21}t(6)-\frac3{196}t(3)^2-\frac12t(2)\zt(\bar3,1)
+\frac14\zt(\bar5,1)\\
-\frac12t(5)\log2+\frac47t(2)t(3)\log2
\end{multline*}
and
\[
dt(2,3,1)=-\frac12t(5)+\frac47t(2)t(3)=t(2,3) .
\]
We note that if $\Z$ denotes the $\Q$-subalgebra of $\R$ spanned by 1
and the multiple zeta values, then no derivation $d$ of $\Z$
with $d\zt=\zt A_-$ can exist.  For if there were such a $d$,
we would have
\[
\zt(2)=\zt A_-(M_{(2,1)})=d\zt(2,1)=d\zt(3)=\zt A_-(M_{(3)})=0 .
\]
\par
If we write $\Z_n$ for the $\Q$-subspace of $\Z$ spanned
by multiple zeta values of weight $n$, then it is generally believed
that $\dim\Z_n=P_n$, where $P_n$ is the $n$th Padovan number (i.e., $P_1=0$,
$P_2=P_3=1$, and $P_n=P_{n-2}+P_{n-3}$ for $n>3$).
It is unlikely that this will be proved soon, because of the
difficulty of showing, for example, that $\zt(3)^2$ is not a 
rational multiple of $\zt(6)$.
Nevertheless, it has been known for some time that 
$\dim\Z_n$ is {\it at most} $P_n$, and indeed
F. Brown \cite{Br} proved that the set of cardinality $P_n$
suggested by the author in \cite{H1997}--the multiple zeta 
values of weight $n$ whose exponent strings have only 2's
and 3's--spans $\Z_n$.
If $\T_n$ denotes the $\Q$-subspace of $\T$ spanned by
all multiple $t$-values of weight $n$, we offer the following 
conjecture.
\begin{conj} 
\label{fib}
For $n\ge 2$, $\dim\T_n=F_n$, the $n$th Fibonacci number
(defined by $F_1=F_2=1$ and $F_{n+2}=F_{n+1}+F_n$ for $n\ge 1$).
\end{conj}
\par\noindent
From Appendix A it follows that $\dim\T_n\le F_n$ for $n\le 7$.
If we let $\A_n$ be the $\Q$-subspace of $\R$ generated by all the 
alternating multiple zeta values of weight $n$, then it is
a long-standing conjecture that $\dim\A_n=F_{n+1}$, and it is
known that $\dim\A_n\le F_{n+1}$ (see \cite[Thm. 13.2.1]{Z}).
\par
We note that Conjecture \ref{der} implies the surjectivity of 
$d:\T_{n+1}\to\T_n$, since any multiple $t$-value $t(i_1,\dots,i_k)\in\T_n$
would be the image under $d$ of $t(i_1,\dots,i_k,1)\in\T_{n+1}$.
Hence Conjectures \ref{der} and \ref{fib} together imply that
$\dim(\T_n\cap\ker d)=F_{n-2}$ for $n\ge 3$.
A more explicit form of this comes from the following conjecture
of B. Saha \cite{Sa}, which is similar to the author's conjecture
on $\Z_n$ in \cite{H1997}.
\begin{conj}[B. Saha] For $n\ge 2$, $\T_n$ has basis
\label{sah}
\[
C_n=\{t(a_1+1,a_2,\dots,a_r)\ |\ a_1+\dots+a_r=n-1,\ a_i\in\{1,2\}\} .
\]
\end{conj}
Conjecture \ref{sah} is consistent with the relations in Appendix A;
in fact, in Appendix B we express all multiple $t$-values through
weight 7 in terms of elements of $C_n$.
Note that for $n\ge 3$, $C_n$ is the union of disjoint subsets
$C_n^{(j)}=\{t(a_1+1,a_2,\dots,a_r)\in C_n\ |\ a_r=j\}$, $j=1,2$; further,
\[
t(a_1+1,a_2,\dots,a_r)\in C_n^{(j)}\quad\text{implies}\quad
t(a_1+1,a_2,\dots,a_{r-1})\in C_{n-j} .
\]
This gives an inductive proof that $|C_n|=F_n$, but note also that
$C_n\cap \ker d=C_n^{(2)}$.
\par
The relation between the $\Q$-subalgebras $\T$ and $\Z$
of $\R$ is far from clear.  Since $\zt(n)$ is a rational multiple
of $t(n)$ for $n\ge 2$, any rational polynomial in the ordinary zeta
values is in $\T$.
In particular, $\Z_n\subseteq\T_n$ for $n\le 7$.
M. Kaneko and H. Tsumura \cite{KTs} make the following conjecture,
which implies that $\Z\subset\T$.
\begin{conj}[M. Kaneko and H. Tsumura] A basis for $\Z_n$ 
is given by the set of elements $t(2)^kt(n_1,\dots,n_r)$
with all the $n_i$ odd and at least 3, and $n_1+\dots+n_r=n-2k$.
\end{conj}
\par\noindent
We remark that the number of elements in Kaneko and Tsumura's conjectural
basis for $\Z_n$ is the coefficient of $t^n$ in
\[
\frac1{1-t^2}\cdot\frac1{1-t^3-t^5-t^7-\cdots}=
\frac1{1-t^2}\cdot\frac1{1-\frac{t^3}{1-t^2}}=\frac1{1-t^2-t^3},
\]
and this latter function is well-known as the generating function
of the Padovan numbers $P_n$ mentioned above.
\section{Multiple $t$-values as homomorphic images}
The following result for multiple $t$-values parallels the
result \cite[Thm. 4.2]{H1997} for multiple zeta values.
\begin{thm}
\label{tquas}
There is an algebra homomorphism $\th:\H^0\to\R$ sending 1 to 1 and
$z_{i_1}\cdots z_{i_k}$ to $t(i_1,\dots,i_k)$
for all positive-integer strings $(i_1,\dots,i_k)$ with $i_1>1$.
\end{thm}
\begin{proof}
The point is that the recursive rule (\ref{recur}) for the words
in the $z$'s corresponds to the rules for multiplying the $t(i_1,\dots,i_k)$,
e.g.,
\[
t(2)t(3,1)=t(2,3,1)+t(3,2,1)+t(3,1,2)+t(5,1)+t(3,3) .
\]
\end{proof}
\par
The next result corresponds to the symmetric-sum theorem \cite[Thm. 2.2]{H1992}
for multiple zeta values.
\begin{thm}
\label{perth}
Let $i_1,\dots,i_k$ be integers all 2 or greater.  If the symmetric
group $S_k$ acts on strings of length $k$ by permutation, then
\[
\sum_{\si\in S_k}t(\si\cdot(i_1,\dots,i_k))=
\sum_{\B=\{B_1,\dots,B_l\}\in\Pi_k}(-1)^{k-l}c(\B)\prod_{s=1}^lt
\left(\sum_{j\in B_s} i_j\right) ,
\]
where $\Pi_k$ is the set of partitions of the set $\{1,2,\dots k\}$
and 
\[
c(\B)=(\card B_1-1)!(\card B_2-1)!\cdots (\card B_l-1)!
\]
for a partition $\B\in\Pi_k$ with blocks $B_1,\dots,B_l$.
\end{thm}
\begin{proof}
The following identity holds in $\QS$ \cite[Thm. 2.3]{H2015}:
\begin{equation*}
\sum_{\si\in S_k} M_{\si\cdot I}=\sum_{\B=\{B_1,\dots,B_l\}\in\Pi_k}
(-1)^{k-l}c(\mathcal B)M_{(b_1)}M_{(b_2)}\cdots M_{(b_l)} ,
\end{equation*}
where $I=(i_1,\dots,i_k)$ and $b_s=\sum_{j\in B_s}i_j$.
Apply the homomorphism $\th$ to obtain the conclusion.
\end{proof}
If we take $i_1=\dots=i_k=n$ in this result, we get an expression
for multiple $t$-values of repeated arguments.  With a little 
work, we can state it in terms of integer rather than set partitions.
\begin{cor}
\label{waus}
If $n\ge 2$, then
\[
t(\{n\}_k)=\sum_{\la\vdash k}\frac{(-1)^{k-\ell(\la)}}{m_1(\la)!1^{m_1(\la)}
m_2(\la)!2^{m_2(\la)}\cdots}\prod_{j=1}^{\ell(\la)}t(n\la_j),
\]
where $\ell(\la)$ is the number of parts of the partition $\la$
and $m_i(\la)$ is the multiplicity of $i$ in $\la$.
\end{cor}
\begin{proof}
Set $i_1=\dots=i_k=n$ in Theorem \ref{perth} to get
\[
k!t(\{n\}_k)=
\sum_{\substack{\text{part. $\{B_1,\dots,B_l\}$}\\ \text{of $\{1,\dots,k\}$}}}
(-1)^{k-l}(\la_1-1)!\cdots (\la_l-1)!t(n\la_1)\cdots t(n\la_l) ,
\]
where we write $\la_i$ for $\card B_i$.  Now the number of 
partitions $\{B_1,\dots,B_l\}$ of the set $\{1,\dots,k\}$ corresponding
to a partition $\la=(\la_1,\dots,\la_l)$ of $k$ is
\[
\frac1{m_1(\la)m_2(\la)\cdots}\binom{k}{\la_1}\binom{k-\la_1}{\la_2}\cdots
=\frac1{m_1(\la)m_2(\la)\cdots}\frac{k!}{\la_1!\la_2!\cdots \la_l!},
\]
so
\begin{multline*}
t(\{n\}_k)=\sum_{\la\vdash k}
\frac{(-1)^{k-l}(\la-1)!\cdots (\la_l-1)!}
{m_1(\la)m_2(\la)\cdots \la_1!\cdots \la_l!}
t(n\la_1)\cdots t(n\la_l)=\\
\sum_{\la\vdash k}
\frac{(-1)^{k-l}}{m_1(\la)!1^{m_1(\la)}m_2(\la)!2^{m_2(\la)}\cdots}
t(n\la_1)\cdots t(n\la_l) ,
\end{multline*}
and the result follows.
\end{proof}
An alternative way to express the preceding result is as follows.
Let $P_k(x_1,\dots,x_k)$ be the polynomial that expresses the
$k$th elementary symmetric function $e_k$ in terms of the power
sums $p_1,\dots,p_k$, i.e.,
\begin{equation}
\label{elesy}
e_k=P_k(p_1,p_2,\dots,p_k)
\end{equation}
(cf. \cite[Eqs. ($2.14^\prime$)]{Mac}).  Then
\begin{equation}
\label{tkrep}
t(\{n\}_k)=P_k(t(n),t(2n),\dots,t(kn)).
\end{equation}
If $n$ is even, Corollary \ref{waus} and Eq. (\ref{teur}) imply that
$t(\{n\}_k)$ is a rational multiple of $\pi^{nk}$.  As we see below,
for small even values of $n$ there are effective formulas for this
rational multiple.
\par
As shown in \cite{H1997}, the homomorphism $\zt:\H^0\to\R$ can be
extended to $\H^1$ by defining $\zt(z_1)=\ga$ (Euler's constant).
This extension has the property that it sends the generating function
$H(x)$ of the complete symmetric functions to $\Ga(1-x)$.
For $\th$ we have the following.
\begin{thm}
\label{gft}
The homomorphism $\th:\H^0\to\R$ can be extended to a homomorphism
$\th:\H^1\to\R$ such that
\[
\th(H(x))={\pi}^{-\frac12}e^{-\frac{\ga x}2}\Ga\left(\frac{1-x}2\right) .
\]
\end{thm}
\begin{proof}
In view of Theorem \ref{quas} above, it suffices to define
$\th(z_1)$, which we set equal to $\log 2$.
Then because
\begin{equation}
\label{hpsi}
-\frac12\psi\left(\frac12-\frac{x}2\right)=\frac{\ga}2+\log 2
+\sum_{i\ge 2} t(i)x^{i-1}
\end{equation}
(for which see \cite[Eqs. (8.370,8.373)]{GR}), where $\psi$ is the 
logarithmic derivative of the gamma function, we have
\[
\th(P(x))=-\frac{\ga}2-\frac12\psi\left(\frac{1-x}2\right) ,
\]
where $P(x)=\sum_{i\ge 1}p_ix^{i-1}$.  Since $P(x)$ is the logarithmic
derivative of $H(x)$, the conclusion follows.
Cf. \cite[Eq. (0.7b)]{Che}.
\end{proof}
If we extend the notation $t(i_1,\dots,i_k)$ to all strings of positive
integers $i_1,\dots,i_k$ by letting 
\[
t(i_1,\dots,i_k)=\th(M_{i_k,\dots,i_1}),
\]
then Theorem \ref{perth} and Corollary \ref{waus} are true without
restrictions on the positive integers involved.  For example,
we then have
\begin{equation}
\label{tone}
t(\{1\}_k)=\sum_{\la\vdash k}
\frac{(-1)^{k-l}}{m_1(\la)!1^{m_1(\la)}m_2(\la)!2^{m_2(\la)}\cdots}
t(\la_1)\cdots t(\la_l) .
\end{equation}
We note that $M_{1,\dots,1}$ ($n$ repetitions of 1) is the elementary
symmetric function $e_n$, so 
\[
1+\sum_{k=1}^\infty t(\{1\}_k)x^k=\th E(x)=\th\left(\frac1{H(-x)}\right)
=\sqrt{\pi}e^{-\frac{\ga x}2}\Ga\left(\frac{1+x}2\right)^{-1}
\]
by Theorem \ref{gft}.
\par
We now return to multiple $t$-values of the form $t(n,n,\dots,n)$
for $n\ge 2$.
\begin{thm}
\label{trep}
For $n\ge 2$, the generating function 
\[
T_n(x)=1+\sum_{k=1}^\infty t(\{n\}_k)x^{kn}
\]
is given by 
\[
T_n(x)=\frac{Z_n(x)}{Z_n\left(\frac{x}{2}\right)},
\]
where $Z_n(x)$ is the corresponding generating function for multiple zeta 
values:
\[
Z_n(x)=1+\sum_{k=1}^\infty \zt(\{n\}_k)x^{kn} .
\]
\end{thm}
\begin{proof}
We start by noting that
\[
H(x)^{-1}=\exp\left(\int_0^x P(t)dt\right)=\exp\left(\sum_{n\ge 1}\frac{p_nx^n}{n}
\right),
\]
so that the generating function $E(x)=\sum_{n\ge 0}e_nx^n$ of the elementary
symmetric functions is
\[
E(x)=H(-x)^{-1}=\exp\left(\sum_{n\ge 1}\frac{(-1)^{n-1}p_nx^n}{n}\right).
\]
Now $t(\{n\}_k)$ is the image under $\th$ of the symmetric function
$\P_n(e_k)$, where $\P_n:\QS\to\QS$ takes any monomial quasi-symmetric 
function $M_{t_1,t_2,\dots,t_p}$ to $M_{nt_1,\dots,nt_p}$.
Then $T_n(x)$ is the image under $\th\P_n$ of $E(x^n)$, and so can 
be written
\begin{multline*}
\exp\left(\sum_{k\ge 1}\frac{(-1)^{k-1}t(nk)x^{kn}}{k}\right)=
\exp\left(\sum_{k\ge 1}\frac{(-1)^{k-1}(1-2^{-nk})\zt(nk)x^{kn}}{k}\right)\\
=\exp\left(\sum_{k\ge 1}\frac{(-1)^{k-1}\zt(nk)x^{kn}}{k}\right)
\exp\left(-\sum_{k\ge 1}\frac{(-1)^{k-1}\zt(nk)x^{kn}}{k2^{kn}}\right)=
\frac{Z_n(x)}{Z_n\left(\frac{x}{2}\right)} .
\end{multline*}
\end{proof}
From this result we can deduce the identities in Eqs. (\ref{krep}) above,
using the known results about multiple zeta values.  From \cite{BBB}
we have 
\[
Z_{2m}(x)=\frac1{(i\pi x)^m}\prod_{j=1}^m\sin(e^{\frac{(2j-1)\pi i}{2m}}\pi x) ,
\]
so it follows from Theorem \ref{trep} that
\begin{equation}
\label{t2m}
T_{2m}(x)=\prod_{j=1}^m\cos\left(e^{\frac{(2j-1)\pi i}{2m}}\frac{\pi x}2\right) .
\end{equation}
Hence
\begin{align*}
T_2(x)&=\cos\left(e^{\frac{\pi i}2}\frac{\pi x}2\right)=
\cosh\left(\frac{\pi x}2\right)\\
T_4(x)&=\cos\left(e^{\frac{\pi i}4}\frac{\pi x}2\right)
\cos\left(e^{\frac{3\pi i}4}\frac{\pi x}2\right)=
\frac12\left[\cos\left(\frac{\pi x}{\sqrt2}\right)+
\cosh\left(\frac{\pi x}{\sqrt2}\right)\right]\\
T_6(x)&=\cos\left(e^{\frac{\pi i}6}\frac{\pi x}2\right)
\cos\left(e^{\frac{\pi i}2}\frac{\pi x}2\right)
\cos\left(e^{\frac{5\pi i}6}\frac{\pi x}2\right)=\\
&\frac14\left[1+\cos\left(e^{\frac{\pi i}6}\pi x\right)+
\cos\left(e^{\frac{\pi i}2}\pi x\right)+
\cos\left(e^{\frac{5\pi i}6}\pi x\right)\right],
\end{align*}
from which Eqs. (\ref{krep}) follow.
\par
For $m=4$ we have
\begin{multline*}
T_8(x)=\prod_{j=1}^4\cos\left(e^{\frac{(2j-1)\pi i}8}\frac{\pi x}2\right)=\\
\frac18\left[\Phi(\al\pi x)+\Phi(\be\pi x)+\Phi(e^{\frac{\pi i}4}\al\pi x)
+\Phi(e^{\frac{\pi i}4}\be\pi x)\right],
\end{multline*}
where $\al=\sqrt{1+\frac1{\sqrt2}}$, $\be=\sqrt{1-\frac1{\sqrt2}}$,
and $\Phi(x)=\cos x+\cosh x$.
From this follows
\begin{equation}
\label{teight}
t(\{8\}_k)=\frac{\pi^{8k}}{2^{2k+1}(8k)!}[(3+2\sqrt2)^{2k}+(3-2\sqrt2)^{2k}] .
\end{equation}
An equivalent identity is given by C-L. Chung \cite{Chung}.
Eq. (\ref{teight}) may be compared to the corresponding formula from
\cite{BBB}:
\[
\zt(\{8\}_k)=\frac{8(2\pi)^{8k}}{2^{2k+1}(8k+4)!}[(3+2\sqrt2)^{2k+1}+(3-2\sqrt2)^{2k+1}] .
\]
\par
For $m=5$ we have
\begin{multline*}
T_{10}(x)=\prod_{j=1}^5\cos\left(\rho^{2j-1}\frac{\pi x}2\right)=\\
\frac1{16}\left[1+\sum_{j=1}^5\left[\cos(\rho^{2j-1}\pi x)
+\cos(\rho^{2j-1}\si\pi x)+\cos(\rho^{2j-1}\tau\pi x)\right]\right],
\end{multline*}
where $\rho=e^{\frac{\pi i}{10}}$, $\si=\frac12(\sqrt5-1)$, and 
$\tau=\frac12(\sqrt5+1)$.  From this it follows that
\begin{equation}
\label{tten}
t(\{10\}_k)=\frac{5\pi^{10k}(L_{10k}+1)}{16(10k)!} ,
\end{equation}
where $L_n$ is the $n$th Lucas number.
This corresponds to the result of \cite{BBB} that
\[
\zt(\{10\}_k)=\frac{10(2\pi)^{10k}(L_{10k+5}+1)}{(10k+5)!} .
\]
\par
For $m=6$ we have
\begin{multline*}
T_{12}(x)=\prod_{j=1}^6\cos\left(e^{\frac{(2j-1)\pi i}{12}}\frac{\pi x}2\right)=
\frac1{16}\left[1+\Phi(\xi\pi x)+\Phi(\xi^3\pi x)+\Phi(\xi^5\pi x)\right]\\
+\frac1{32}\sum_{j=0}^2\left[\Phi(\xi^{2j}\sqrt2\pi x)
+\Phi(\xi^{2j}\ga\pi x)+\Phi(\xi^{2j}\de\pi x)\right],
\end{multline*}
where $\xi=e^{\frac{\pi i}{12}}$, $\ga=\sqrt{2+\sqrt3}$, $\de=\sqrt{2-\sqrt3}$,
and $\Phi$ is as above.
From this follows
\begin{equation}
\label{ttwelve}
t(\{12\}_k)=\frac{3\pi^{12k}}{16(12k)!}[(2+\sqrt3)^{6k}+(2-\sqrt3)^{6k}+
2^{6k}+(-1)^k2].
\end{equation}
This may be compared to the corresponding result in \cite{BBB}:
\[
\zt(\{12\}_k)=\frac{12(2\pi)^{12k}}{(12k+6)!}[(2+\sqrt3)^{6k+3}
+(2-\sqrt3)^{6k+3}+2^{6k+3}] .
\]
Z. Shen and L. Jia \cite[Thm. 2]{SJ2017} establish a general result for 
$t(\{2m\}_n)$ that is somewhat less explicit than the cases given here.
\par
Multiple $t$-star values are defined by Eq. (\ref{mtsv}) above.
We have the following result, which corresponds to \cite[Thm. 2.1]{H1992}.  
The notation is as in Theorem \ref{perth} above.
\begin{thm}
\label{adelaide}
Let $i_1,\dots,i_k$ be integers all 2 or greater.  If the symmetric
group $S_k$ acts on strings of length $k$ by permutation, then
\[
\sum_{\si\in S_k}t^{\star}(\si\cdot(i_1,\dots,i_k))=
\sum_{\B=\{B_1,\dots,B_l\}\in\Pi_k}c(\B)\prod_{s=1}^lt
\left(\sum_{j\in B_s} i_j\right) .
\]
\end{thm}
\begin{proof} See \cite[Thm. 4.1]{H2015}.
\end{proof}
If we take $i_1=i_2=\dots=i_k=n$ in this theorem we get a formula
for $t^{\star}(\{n\}_k)$, $n\ge 2$, comparable to Corollary \ref{waus} above.
\begin{cor}
\label{saus}
If $n\ge 2$, then
\[
t^{\star}(\{n\}_k)=\sum_{\la\vdash n}\frac1{m_1(\la)1^{m_1(\la)}m_2(\la)2^{m_2(\la)}\cdots}
\prod_{j=1}^{\ell(\la)}t(n\la_j) .
\]
\end{cor}
As with Corollary \ref{waus}, this result has an expression involving
the symmetric functions.  Let $Q_k(x_1,\dots,x_k)$ be the polynomial
expressing the complete symmetric function $h_k$ in terms of power
sums $p_1,p_2,\dots,p_k$.  Then
\[
t^{\star}(\{n\}_k)=Q_k(t(n),t(2n),\dots,t(kn)) .
\]
We can extend $t^{\star}$ to any string of positive integers using
Theorem \ref{gft}.  Then there is an analogue of Eq. (\ref{tone}) for
$t^{\star}$.
\[
t^{\star}(\{1\}_k)=\sum_{\la\vdash k}
\frac1{m_1(\la)!1^{m_1(\la)}m_2(\la)!2^{m_2(\la)}\cdots}
t(\la_1)\cdots t(\la_l) .
\]
We can also express the numbers $t^{\star}(\{n\}_k)$ using generating functions.
Here we have a result similar to Theorem \ref{trep}.
\begin{thm}
For integers $n\ge 2$,
\[
1+\sum_{k=1}^\infty t^{\star}(\{n\}_k)x^{kn}=\frac{Z_n(e^{\frac{\pi i}n}\frac{x}2)}
{Z_n(e^{\frac{\pi i}n}x)} .
\]
\end{thm}
\begin{proof}
We note that $t^{\star}(\{n\}_k)$ is the image under $\th\P_n$ of $h_k$,
so that
\begin{multline*}
1+\sum_{k=1}^\infty t^{\star}(\{n\}_k)x^{kn}=\th\P_n(H(x))
=\exp\left(\sum_{k\ge 1}\frac{t(nk)x^{kn}}{k}\right)\\
=\exp\left(\sum_{k\ge 1}\frac{(1-2^{-nk})\zt(nk)x^{kn}}{k}\right).
\end{multline*}
Now proceed as in the proof of Theorem \ref{trep}.
\end{proof}
In the case where $n$ is an even integer $2m$, we can express 
$t^{\star}(\{n\}_k)$ as a rational multiple of $\pi^{nk}$.
Similar results have been proved by Shen and Jia \cite{SJ2017}, and
by Chung \cite{Chung}.
\begin{thm}
\label{eulern}
For positive integers $m$ and $k$,
\[
t^{\star}(\{2m\}_k)=\frac{(-1)^{mk}}{(2mk)!}\left(\frac{\pi}2\right)^{2mk}
\sum_{\substack{n_1+\dots+n_m=mk\\n_j\ge 0}}\binom{2mk}{2n_1\cdots 2n_m}
\prod_{j=1}^m e^{\frac{2\pi i}{m}(j-1)n_j}E_{2n_j} .
\]
\end{thm}
\begin{proof}
We follow Muneta's proof \cite{M} of the corresponding result for
multiple zeta-star values.
Using the infinite product for cosine, we have
\[
\sec\left(\pi e^{\frac{\pi i k}{m}}x\right)=
\prod_{h=1}^\infty\left(1-\frac{e^{\frac{2\pi i k}{m}}(2x)^2}{(2h-1)^2}\right)^{-1} .
\]
From this follows
\begin{equation}
\label{prod}
\prod_{k=0}^{m-1}\sec\left(\pi e^{\frac{\pi i k}{m}}x\right)=
\prod_{h=1}^\infty\left(1-\frac{(2x)^{2m}}{(2h-1)^{2m}}\right)^{-1}
\end{equation}
since 
\[
(1-e^{\frac{2\pi i}{m}}u)(1-e^{\frac{4\pi i}{m}}u)\cdots(1- e^{\frac{2(m-1)\pi i}{m}}u)
=1-u^m .
\]
Now the right-hand side of Eq. (\ref{prod}) can be expanded as
\begin{multline}
\label{serex}
1+\sum_{h_1\ge 1}\frac{(2x)^{2m}}{(2h_1-1)^{2m}}+\sum_{h_1\ge h_2\ge 1}
\frac{(2x)^{4m}}{(2h_1-1)^{2m}(2h_2-1)^{2m}}+\cdots\\
=1+\sum_{k=1}^\infty t^{\star}(\{2m\}_k)(2x)^{2km} .
\end{multline}
On the other hand, using the Maclaurin series for secant we can expand
the left-hand side of Eq. (\ref{prod}) as
\begin{multline*}
\prod_{k=0}^{m-1}\left(\sum_{j=0}^\infty \frac{(-1)^jE_{2j}}{(2j)!}
\pi^{2j}e^{\frac{2\pi i kj}{m}}x^{2j}\right)=\\
\sum_{n=0}^\infty \sum_{j_0+\dots+j_{m-1}=nm}
\frac{(-1)^{nm}E_{j_0}E_{j_1}\cdots E_{j_{m-1}}}{(2j_0)!(2j_1)!\cdots (2j_{m-1})!}
e^{\frac{2\pi i j_1}{m}+\frac{4\pi i j_2}{m}+\frac{2(m-1)\pi i j_{m-1}}{m}}(\pi x)^{2mn},
\end{multline*}
where we have used the fact that only powers of $x^{2m}$ appear in the 
expansion.  The latter expression can be written as
\[
(-1)^{nm}\sum_{n=0}^\infty \frac{(\pi x)^{2nm}}{(2nm)!}
\sum_{n_1+\dots+n_m=nm}\binom{2nm}{2n_1\cdots 2n_m}\prod_{j=1}^m
e^{\frac{2\pi i}{m}(j-1)n_j}E_{2n_j} ,
\]
and comparing coefficients with Eq. (\ref{serex}) gives the conclusion.
\end{proof}
\section{Multiple $t$-values and alternating multiple zeta values}
\par
Following \cite{H2000}, let $\E_{2}$ be the underlying rational vector
space of the noncommutative polynomial algebra on generators $z_{n,p}$,
$n\in\{1,2,\dots\}$ and $p\in\{0,1\}$, with the product $*$ defined 
recursively by
\begin{equation}
\label{e2rec}
aw_1*bw_2=a(w_1*bw_2)+b(aw_1*w_2)+(a\circ b)(w_1*w_2)
\end{equation}
for words $w_1,w_2$ and letters $a,b$.  
Here the operation $\circ$ is given by
\[
z_{n_1,p_1}\circ z_{n_2,p_2}=z_{n_1+n_2,p_1+p_2} ,
\]
where addition in the second subscript is taken mod 2.
Then if $\E_2^0$ is the subspace generated by all words that do not
begin with $z_{1,0}$, $(\E_2^0,*)$ is a subalgebra of $(\E_2,*)$.
There is a homomorphism $Z:\E_2^0\to\R$ defined by
\[
Z(z_{n_1,p_1}\cdots z_{n_k,p_k})=
\sum_{m_1>\dots>m_k\ge 1}\frac{(-1)^{m_1p_1+\dots+m_kp_k}}{m_1^{n_1}\cdots m_k^{n_k}} .
\]
The series on the right-hand side of the preceding equation is
an alternating or ``colored'' multiple zeta value.
The usual notation for multiple zeta values can be extended to
such quantities by using an upper bar, e.g., $\zt(\bar 3,2,\bar 1)$
denotes $Z(z_{3,1}z_{2,0}z_{1,1})$.
\par
For any $u\in\E_2$ and monomial $w$ of $\E_2$, let 
$\coeff_u(w)$ denote the coefficient of $w$ in $u$.
Call an element $u\in\E_2$ totally symmetric if 
\[
\coeff_u(z_{n_1,p_1}\cdots z_{n_k,p_k})=
\coeff_u(z_{n_1,0}\cdots z_{n_k,0}) 
\]
for any monomial $z_{n_1,p_1}\cdots z_{n_k,p_k}$ of $\E_2$,
and totally antisymmetric if
\[
\coeff_u(z_{n_1,p_1}\cdots z_{n_k,p_k})=
(-1)^{p_1+\dots+p_k}\coeff_u(z_{n_1,0}\cdots z_{n_k,0}) .
\]
Let $\E_2^S$ be the vector space of totally symmetric elements of $\E_2$,
and $\E_2^A$ the vector space of totally antisymmetric elements.
\begin{thm}
$\E_2^S$ and $\E_2^A$ are subalgebras of $(\E_2,*)$.
Further, both are isomorphic to $(\H^1,*)$.
\end{thm}
\begin{proof}
We use the result of \cite{H2000} that $\E_2$ is isomorphic to 
a subring of the power series ring $\Q[[t_1,t_2,\dots]]$ via the
map $\phi:\E_2\to\Q[[t_1,t_2,\dots]]$ given by
\[
\phi(z_{n_1,p_1}\cdots z_{n_k,p_k})=\sum_{m_1>\dots>m_k\ge 1}(-1)^{m_1p_1+\cdots+
m_kp_k}t_{m_1}^{n_1}\cdots t_{m_k}^{n_k} .
\]
Now let $u\in\E_2^S$, and consider $\phi(u)\in\Q[[t_1,t_2,\dots]]$.
For any monomial $z_{n_1,p_1}\cdots z_{n_k,p_k}$ occurring in $u$,
\begin{align*}
\coeff_{\phi(u)}(t_{m_1}^{n_1}\cdots t_{m_k}^{n_k})&=
\sum_{p_1=0}^1\cdots\sum_{p_k=0}^1(-1)^{m_1p_1+\dots+m_kp_k}
\coeff_u(z_{n_1,0}\cdots z_{n_k,0})\\
&=\prod_{i=1}^k(1+(-1)^{m_i})\coeff_u(z_{n_1,0}\cdots z_{n_k,0})\\
&=\begin{cases} 2^k\coeff_u(z_{n_1,0}\cdots z_{n_k,0}),
&\text{if all the $m_i$ are even,}\\
0,&\text{otherwise.}\end{cases}
\end{align*}
Thus $\phi$ sends $\E_2^S$ to the subring of $\Q[[t_1,t_2,\dots]]$
generated by the power series
\[
\sum_{m_1>m_2>\dots>m_k\ge 1,\ \text{$m_i$ even}} t_{m_1}^{n_1}\cdots t_{m_k}^{n_k} .
\]
This is evidently isomorphic to $\QS$.
\par
Now suppose $u\in\E_2^A$.  For any monomial $z_{n_1,p_1}\cdots z_{n_k,p_k}$
occurring in $u$, we have
\begin{align*}
\coeff_{\phi(u)}(t_{m_1}^{n_1}&\cdots t_{m_k}^{n_k})=
\sum_{p_1=0}^1\cdots\sum_{p_k=0}^1(-1)^{m_1p_1+\dots+m_kp_k}
\coeff_u(z_{n_1,p_1}\cdots z_{n_k,p_k})\\
&=\sum_{p_1=0}^1\cdots\sum_{p_k=0}^1(-1)^{(m_1+1)p_1+\dots+(m_k+1)p_k}
\coeff_u(z_{n_1,0}\cdots z_{n_k,0})\\
&=\prod_{i=1}^k(1+(-1)^{m_i+1})\coeff_u(z_{n_1,0}\cdots z_{n_k,0})\\
&=\begin{cases} 2^k\coeff_u(z_{n_1,0}\cdots z_{n_k,0}),
&\text{if all the $m_i$ are odd,}\\
0,&\text{otherwise.}\end{cases}
\end{align*}
Thus $\phi$ sends $\E_2^A$ to the subring of $\Q[[t_1,t_2,\dots]]$
generated by power series
\[
\sum_{m_1>m_2>\dots>m_k\ge 1,\ \text{$m_i$ odd}} t_{m_1}^{n_1}\cdots t_{m_k}^{n_k} ,
\]
which is also isomorphic to $\QS$.
\end{proof}
We can define functions $S:\H^1\to\E_2^S$ and $A:\H^1\to\E_2^A$ by
\begin{align*}
S(z_{n_1}\cdots z_{n_k})&=\sum_{p_1,\dots,p_k\in\{0,1\}}z_{n_1,p_1}\cdots z_{n_k,p_k}\\
A(z_{n_1}\cdots z_{n_k})&=\sum_{p_1,\dots,p_k\in\{0,1\}}(-1)^{p_1+\dots+p_k}
z_{n_1,p_1}\cdots z_{n_k,p_k} .
\end{align*}
By the preceding proof we have
\begin{equation}
\label{phis}
\phi S(z_{n_1}\cdots z_{n_k})=2^k\sum_{m_1>\dots>m_k\ge 1,\ \text{$m_i$ even}}
t_{m_1}^{n_1} \cdots t_{m_k}^{n_k}
\end{equation}
and
\begin{equation}
\label{phia}
\phi A(z_{n_1}\cdots z_{n_k})=2^k\sum_{m_1>\dots>m_k\ge 1,\ \text{$m_i$ odd}}
t_{m_1}^{n_1} \cdots t_{m_k}^{n_k}.
\end{equation}
\par
Now consider the homomorphism $\ev:\Q[[t_1,t_2,\dots]]\to\R$ sending
$t_j$ to $\frac1j$.  Of course this doesn't make sense on all of
$\Q[[t_1,t_2,\dots]]$, but by Eq. (\ref{phia}) it does send 
$\phi A(z_{n_1}\cdots z_{n_k})$ to $2^kt(n_k,\dots,n_1)$ if $n_k>1$.  
Hence we have the following.
\begin{cor}
\label{expr}
For positive integers $a_1,\dots,a_k$ with $a_1\ge 2$,
\[
t(a_1,\dots,a_k)=\frac1{2^k}\sum_{\ep_1,\dots,\ep_k=\pm 1}\ep_1\cdots\ep_k
\zt(\ep_1\op a_1,\dots,\ep_k\op a_k) ,
\]
where the sum is over the $2^k$ $k$-tuples $(\ep_1,\dots,\ep_k)$ with
each $\ep_i\in\{1,-1\}$, and $\op$ is defined by $1\op i=i$ and
$-1\op i=\bar i$.
\end{cor}
For double $t$-values this is
\begin{equation}
\label{k2}
t(a,b)=\frac14(\zt(a,b)-\zt(a,\bar b)-\zt(\bar a,b)+\zt(\bar a,\bar b)) ,
\end{equation}
as stated in \cite{KT} and \cite{NT}.
The preceding result expresses a $t$-value of depth $k$ as a sum
of $2^k$ alternating multiple zeta values.  Actually one can do somewhat
better:  it is possible to write such a multiple $t$-value as a sum
of $2^{k-1}$ alternating multiple zeta values as follows.
We require a bit of additional notation.  For $p\le k$, let
$L_p\zt(i_1,\dots,i_k)$ be the sum of all $\binom{k}{p}$ alterating multiple 
zeta values in which the upper bar is applied to exactly $p$ of the positive 
integers $i_j$, e.g.,
\[
L_2\zt(i_1,i_2,i_3)=\zt(\bar i_1,\bar i_2,i_3)+\zt(\bar i_1,i_2,\bar i_3)
+\zt(i_1,\bar i_2,\bar i_3) .
\]
Then we have the following result.
\begin{cor}
\label{expr2}
For positive integers $a_1,\dots,a_k$ with $a_1\ge 2$,
\begin{multline*}
t(a_1,\dots,a_k)=\left(\frac1{2^{k-1}}-\frac1{2^{a_1+\dots+a_k}}\right)
\zt(a_1,\dots,a_k)\\
+\frac1{2^{k-1}}\sum_{\text{$2\le p\le k$ even}}L_p\zt(a_1,\dots,a_k) .
\end{multline*}
\end{cor}
\begin{proof}
Using the notation just introduced, the previous corollary can be stated
\begin{equation}
\label{asy}
2^{-k}\sum_{p=0}^k(-1)^pL_p\zt(a_1,\dots,a_k)=t(a_1,\dots,a_k).
\end{equation}
Now apply $\ev$ to Eq. (\ref{phis}) above to get
\[
2^{-k}\sum_{p=0}^kL_p\zt(a_1,\dots,a_k)
=\sum_{\substack{n_1>\dots>n_k\ge 1\\ \text{$n_i$ even}}}\frac1{n_1^{a_1}\cdots n_k^{a_k}}
=2^{-a_1-\dots-a_k}\zt(a_1,\dots,a_k),
\]
which when added to Eq. (\ref{asy}) gives the conclusion.
\end{proof}
\par\noindent
We note that already in the case $k=2$, Eq. (\ref{k2}) can be
replaced by
\[
t(a,b)=\left(\frac12-\frac1{2^{a+b}}\right)\zt(a,b)+\frac12\zt(\bar a,\bar b).
\]
\par
The function $Z:\E_2^0\to\R$ can be written as the composition $\ev\phi R$,
where $R$ sends $z_{n_1,p_1}\cdots z_{n_k,p_k}$ to $z_{n_k,p_k}\cdots z_{n_1,p_1}$.
(Note that $\ev$ makes sense on $\phi R(\E_2^0)$.)
This gives us the following result.
\begin{thm}
The image $Z(\E_2^S\cap\E_2^0)\subset\R$ is the set $\Z$ of rational
linear combinations of multiple zeta values, and the image
$Z(\E_2^A\cap\E_2^0)\subset\R$ is the set $\T$ of rational linear
combinations of multiple $t$-values.
\end{thm}
\section{A generating function}
In this section we obtain a formula for the generating function of
height one multiple $t$-values, i.e.,
\[
H(x,y)=\sum_{i,j\ge 1}t(i+1,\{1\}_{j-1})x^iy^j=
\sum_{i,j\ge 1}\th(z_{i+1}z_1^{j-1})x^iy^j .
\]
To this end we introduce some functions indexed by words as follows.
For a nonempty word $w=z_{p_1}z_{p_2}\cdots z_{p_k}$ of $\H^1$
and $r\in\C$, define
\[
\L_r(w)=
\sum_{j_1>\dots>j_k\ge 0}\frac{r^{2j_1+1}}{(2j_1+1)^{p_1}\cdots (2j_k+1)^{p_k}} .
\]
Then $\L_r(w)$ converges if $|r|<1$, and
$\L_1(z_{p_1}\cdots z_{p_k})=t(p_1,\dots,p_k)$ if $p_1>1$.
We have the following result.
\begin{lem}  
\label{difl}
For $k>1$,
\[
\frac{d}{dr}\L_r(z_{p_1}z_{p_2}\cdots z_{p_k})=\begin{cases}
\frac1{r}\L_r(z_{p_1-1}z_{p_2}\cdots z_{p_k}),&\text{if $p_1>1$,}\\
\frac{r}{1-r^2}\L_r(z_{p_2}\cdots z_{p_k}), &\text{if $p_1=1$.}
\end{cases}\]
In the case $k=1$, we have
\[
\frac{d}{dr}\L_r(z_{p_1})=\begin{cases}
\frac1{r}\L_r(z_{p_1-1}),& \text{if $p_1>1$,}\\
\frac1{1-r^2},& \text{if $p_1=1$.}
\end{cases}\]
\end{lem}
\begin{proof}
Evidently
\[
\frac{d}{dr}\L_r(z_{p_1}z_{p_2}\cdots z_{p_k})=
\sum_{j_1>\dots>j_k\ge 0}\frac{r^{2j_1}}{(2j_1+1)^{p_1-1}\cdots (2j_k+1)^{p_k}} ,
\]
which is clearly $\frac1{r}\L_r(z_{p_1-1}z_{p_2}\cdots z_{p_k})$
if $p_1>1$.
Otherwise, it's
\[
\sum_{j_2>j_3>\dots>j_k\ge 0}\frac{r^{2(j_2+1)}+r^{2(j_2+2)}+\cdots}{(2j_2+1)^{p_2}
\cdots (2j_k+1)^{p_k}}=
\frac{r}{1-r^2}\L_r(z_{p_2}\cdots z_{p_k}).
\]
The case $k=1$ follows since
\[
\L_r(z_1)=r+\frac{r^3}3+\frac{r^5}5+\cdots=\int_0^r\frac1{1-t^2}dt .
\]
\end{proof}
We now obtain our formula for $H(x,y)$.
\begin{thm}
\label{hypg32}
\[
H(x,y)=\pFq{3}{2}{\frac{1+y}{2},\frac{1-x}{2},1}{\frac32,\frac{3-x}{2}}{1} .
\]
\end{thm}
\begin{proof}
Define generating functions
\[
H_r(x,y)=\sum_{i,j\ge 1}\L_r(z_{i+1}z_1^{j-1})x^iy^j
\]
and
\[
G_r(y)=\sum_{n\ge 1}\L_r(z_1^n)y^n=\L_r(z_1)y+\L_r(z_1^2)y^2+\L_r(z_1^3)y^3+\cdots .
\]
By Lemma \ref{difl}, the derivative of $G_r$ with respect to $r$ is
\[
\frac1{1-r^2}y+\frac{r}{1-r^2}\L_r(z_1)y^2+\frac{r}{1-r^2}\L_r(z_1^2)y^3
+\cdots
=\frac{y}{1-r^2}+\frac{ry}{1-r^2}G_r(y) .
\]
Hence 
\[
\frac{dG_r}{dr}-\frac{ry}{1-r^2}G_r=\frac{y}{1-r^2} ,
\]
or
\[
\frac{d}{dr}\left((1-r^2)^{\frac{y}2}G_r\right)=\frac{y}{1-r^2}
(1-r^2)^{\frac{y}2}=y(1-r^2)^{\frac{y}2-1} ,
\]
and thus
\begin{multline*}
G_r=(1-r^2)^{-\frac{y}2}\int_0^ry(1-s^2)^{\frac{y}2-1}ds=
y(1-r^2)^{-\frac{y}2}\frac{r}2\int_0^1(1-r^2u)^{\frac{y}2-1}u^{-\frac12}du\\
=ry(1-r^2)^{-\frac{y}2}\pFq{2}{1}{1-\frac{y}2,\frac12}{\frac32}{r^2},
\end{multline*}
where we used Euler's integral formula for ${}_2F_1$.
Then $G_r$ can be written (via \cite[15.16.1]{O})
\begin{multline*}
ry\pFq{2}{1}{\frac{y}2,\frac12}{\frac12}{r^2}
\pFq{2}{1}{1-\frac{y}2,\frac12}{\frac32}{r^2}
=y\sum_{s=0}^\infty\frac{\left(\frac{y}2+\frac12\right)\cdots
\left(\frac{y}2-\frac12+s\right)}{\left(\frac12+1\right)\cdots
\left(\frac12+s\right)}r^{2s+1} .
\end{multline*}
\par
Again using Lemma \ref{difl},
\[
\frac{dH_r}{dr}=\frac1r\sum_{i,j\ge 1}\L_r(z_iz_1^{j-1})x^iy^j=
\frac{x}{r}G_r+\frac{x}{r}H_r
\]
or
\[
\frac{dH_r}{dr}-\frac{x}{r}H_r=xy
\sum_{s=0}^\infty\frac{\left(\frac{y}2+\frac12\right)\cdots
\left(\frac{y}2-\frac12+s\right)}{\left(\frac12+1\right)\cdots
\left(\frac12+s\right)}r^{2s} .
\]
It follows that
\[
\frac{d}{dr}(r^{-x}H_r)=xy\sum_{s=0}^\infty\frac{\left(\frac{y}2
+\frac12\right)\cdots\left(\frac{y}2-\frac12+s\right)}
{\left(\frac12+1\right)\cdots\left(\frac12+s\right)}r^{2s-x}
\]
and thus
\begin{multline*}
H_r=xy\sum_{s=0}^\infty \frac{\left(\frac{y}2+\frac12\right)
\cdots\left(\frac{y}2-\frac12+s\right)}
{\left(\frac12+1\right)\cdots\left(\frac12+s\right)}\frac{r^{2s+1}}{2s-x+1}
=\\
\frac{xyr}{1-x}\sum_{s=0}^\infty
\frac{\left(\frac12+\frac{y}2\right)_s\left(\frac12-\frac{x}2\right)_s}
{\left(\frac32\right)_s\left(\frac32-\frac{x}2\right)_s}r^{2s}
=\frac{xyr}{1-x}
\pFq{3}{2}{\frac{1+y}2,\frac{1-x}2,1}{\frac32,\frac{3-x}2}{r^2}
\end{multline*}
where $(x)_n$ means $x(x+1)\cdots (x+n-1)$.
Now set $r=1$ to obtain the conclusion.
\end{proof}
The coefficient of $xy$ in the definition of $H(x,y)$ is $t(2)$, while 
from Theorem \ref{hypg32} it can be seen to be 
${}_3F_2\left(\frac12,\frac12,1;\frac32,\frac32;1\right)$, so these
must be equal.  
In fact this equality generalizes, but first we need a lemma.
(As above, $\{a\}_n$ means $n$ repetitions of $a$.)
\begin{lem}
\label{difhg}
If
\[
\frac{x}{1-x}\pFq{3}{2}{a,b,\frac{1-x}2}{c,\frac{3-x}2}{1}=
\sum_{j=1}^\infty p_jx^j ,
\]
then
\[
p_n=\pFq{n+2}{n+1}{a,b,\left\{\frac12\right\}_n}{c,\left\{\frac32\right\}_n}{1}.
\]
\end{lem}
\begin{proof}
Note that
\begin{multline*}
\frac{x}{1-x}\pFq{3}{2}{a,b,\frac{1-x}2}{c,\frac{3-x}2}{1}=
\frac{x}{1-x}\sum_{s=0}^\infty\frac{(a)_s(b)_s\left(\frac{1-x}2\right)_s}
{(c)_s\left(\frac{3-x}2\right)_s}\frac1{s!}\\
=x\sum_{s=0}^\infty\frac{(a)_s(b)_s}{(c)_s s!}\frac1{2s+1-x},
\end{multline*}
whose $n$th derivative is seen to be
\[
n!\sum_{s=0}^\infty \frac{(a)_s(b)_s}{(c)_s s!}\frac1{(2s+1-x)^n}+
n!x\sum_{s=0}^\infty \frac{(a)_s(b)_s}{(c)_s s!}\frac1{(2s+1-x)^{n+1}} .
\]
Then
\[
p_n=\frac1{n!}\frac{d^n}{dx^n}\bigg|_{x=0}
\frac{x}{1-x}\pFq{3}{2}{a,b,\frac{1-x}2}{c,\frac{3-x}2}{1}=
\sum_{s=0}^\infty\frac{(a)_s(b)_s}{(c)_ss!}\frac1{(2s+1)^n},
\]
and the conclusion follows since
$\left(\frac12\right)_s/\left(\frac32\right)_s=\frac1{2s+1}$.
\end{proof}
Now we can deduce the following corollary of Theorem \ref{hypg32}.
\begin{cor}
For $n\ge 2$,
$
t(n)={}_{n+1}F_n\left(1,\left\{\frac12\right\}_n;\left\{\frac32\right\}_n;1
\right) .
$
\end{cor}
\begin{proof}
Theorem \ref{hypg32} implies that
\begin{multline*}
t(n)=\text{coefficient of $x^{n-1}y$ in $H(x,y)$}=\\
\text{coefficient of $x^{n-1}$ in}\
\frac{x}{1-x}\pFq{3}{2}{1,\frac12,\frac{1-x}2}{\frac32,\frac{3-x}2}{1} .
\end{multline*}
Now apply Lemma \ref{difhg} with $a=1$, $b=\frac12$, and $c=\frac32$.
\end{proof}
Note also that
\begin{multline*}
t(2,\{1\}_{n-1})=\text{coefficient of $xy^n$ in $H(x,y)$}=\\
\text{coefficient of $y^{n-1}$ in}\
\pFq{3}{2}{1,\frac{1+y}2,\frac12}{\frac32,\frac32}{1} ,
\end{multline*}
and setting $y=1$ in Theorem \ref{hypg32} gives
\[
\sum_{j=1}^\infty t(n,\{1\}_{j-1})=\text{coefficient of $x^{n-1}$ in}\
\frac{x}{1-x}\pFq{3}{2}{1,1,\frac{1-x}2}{\frac32,\frac{3-x}2}{1} .
\]
Applying Lemma \ref{difhg} with $a=b=1$ and $c=\frac32$ to the latter 
gives the following.
\begin{cor} For $n\ge 2$,
\[
\sum_{j=1}^\infty t(n,\{1\}_{j-1})=\pFq{n+1}{n}{1,1,\left\{\frac12\right\}_{n-1}}
{\left\{\frac32\right\}_n}{1} .
\]
\end{cor}
\par\noindent
In particular, for $n=2$ we have
\begin{equation}
\label{cat}
\sum_{j=1}^\infty t(2,\{1\}_{j-1})=\pFq{3}{2}{1,1,\frac12}{\frac32,\frac32}{1}
=2G,
\end{equation}
where $G=\sum_{j=0}^\infty\frac{(-1)^j}{(2j+1)^2}$ is Catalan's constant
(see \cite[Eq. (7.4.4.183a)]{KK}).
\par
Conjecture \ref{der} above implies that
$H(x,y)=e^{y\log2}A(x,y)$, for $A(x,y)$ not depending on $\log2$.
Using the tables of Appendix A, it appears that
\begin{multline*}
A(x,y)=t(2)xy+t(3)x^2y-\frac{t(3)}2xy^2+t(4)x^3y-
\left(\frac{37}{60}t(4)+\frac12\zt(\bar3,1)\right)x^2y^2\\
+\left(\frac{11}{60}t(4)+\frac14\zt(\bar3,1)\right)xy^3+\cdots
\end{multline*}
through degree 4.
\section{Alternating multiple $t$-values}
We can define alternating multiple $t$-values in a way similar to 
alternating multiple zeta values:  for the algebra $\E_2^0$ defined
in \S3, there is a homomorphism $T:\E_2^0\to\R$ defined by
\[
T(z_{n_1,p_1}\cdots z_{n_k,p_k})=\sum_{m_1>\cdots>m_k\ge 1}
\frac{(-1)^{m_1p_1+\dots+m_kp_k}}{(2m_1-1)^{n_1}\cdots (2m_k-1)^{n_k}} .
\]
We write, e.g., $t(\bar3,2,\bar1)$ for $T(z_{3,1}z_{2,0}z_{1,1})$.
Then $t(\bar n)=-\be(n)$, where $\be$ is the Dirichlet beta function
\[
\be(s)=\sum_{n=0}^\infty\frac{(-1)^n}{(2n+1)^s} ,\ \Re(s)>0.
\]
In particular, $t(\bar1)=-\frac{\pi}4$ and $t(\bar2)=-G$.
From the integral representation
\[
t(\bar1,1)=\int_0^1\frac{x\tan^{-1}x}{1+x^2}dx
\]
we have $t(\bar1,1)=\frac{G}2-\frac{\pi}8\log2$.
\par
For any positive integer $n$, define a $\Q$-linear map
$\hat\P_n:\H^1\to\E_2^0$ by $\hat\P_n(1)=1$ and
\[
\hat\P_n(z_{i_1}\cdots z_{i_k})=z_{ni_1,i_1}\cdots z_{ni_k,i_k} ,
\]
where we recall that the second subscript is to be considered mod 2.
\begin{prop} 
$\hat\P_n$ is a homomorphism.
\end{prop}
\begin{proof}
The key point to check is that the definition of $\hat\P_n$ is compatible
with the recursive definitions (\ref{recur}) and (\ref{e2rec}) of the
products in $\H^1$ and $\E_2$ respectively.  Using induction on word length,
we have, for words $w,v$ of $\H^1$,
\begin{align*}
\hat\P_n(z_iw*z_jv)&=\hat\P_n(z_i(w*z_jv)+z_j(z_iw*v)+z_{i+j}(w*v))\\
&=z_{ni,i}(\hat\P_n(w)*\hat\P_n(z_jv))+z_{nj,j}(\hat\P_n(z_iw)*\hat\P_n(v))\\
&+z_{ni+nj,i+j}(\hat\P_n(w)*\hat\P_n(v))\\
&=z_{ni,i}(\hat\P_n(w)*z_{nj,j}\hat\P_n(v))+z_{nj,j}(z_{ni,i}\hat\P_n(w)*\hat\P_n(v))\\
&+z_{ni+nj,i+j}(\hat\P_n(w)*\hat\P_n(v))\\
&=z_{ni,i}\hat\P_n(w)*z_{nj,j}\hat\P_n(v)\\
&=\hat\P_n(z_iw)*\hat\P_n(z_jv) .
\end{align*}
\end{proof}
Then we have a homomorphism $\hat\th_n=T\hat\P_n:\H^1\to\R$, and applying it 
to Eq. (\ref{elesy}) above gives an analogue of Eq. (\ref{tkrep}) for 
alternating multiple $t$-values, i.e.,
\[
t(\{\bar n\}_k)=P_k(t(\bar n),t(2n),t(\overline{3n}),t(4n),\dots),
\]
and this holds for all positive integers $n$.
For example, $t(\bar1,\bar1)=-\frac{\pi^2}{32}$, $t(\bar2,\bar2)=
\frac{G^2}2-\frac{\pi^4}{192}$, and $t(\bar3,\bar3)=-\frac{\pi^6}{30720}$.
In fact, we have an explicit formula for the generating function
of the values $t(\{\bar n\}_k)$ when $n$ is odd, providing a counterpart
to Eq. (\ref{t2m}) above (see also \cite[Eq. (35)]{BBB} for the 
alternating multiple zeta values).
\begin{thm}
\label{barrep}
For nonnegative integers $m$,
\[
1+\sum_{k=1}^\infty t(\{\overline{2m+1}\}_k)x^{(2m+1)k}=
\prod_{j=0}^{2m}\left(1-(-1)^j\sin\left(e^{\frac{\pi j i}{2m+1}}\frac{\pi x}2\right)
\right)^{\frac12} .
\]
\end{thm}
\begin{proof}
As in \S2, $P$, $H$, and $E$ are the generating functions of power-sum,
complete, and elementary symmetric functions respectively.
Starting with
\[
\frac{\pi}4\tan\frac{\pi x}2-\frac{\pi}4\sec\frac{\pi x}2=
\sum_{\text{$k\ge 0$ even}}^\infty t(\overline{k+1})x^k
+\sum_{\text{$k\ge 1$ odd}}^\infty t(k+1)x^k
\]
(which follows from Eq. (\ref{tangen}) above and \cite[Eq. (23.2.22)]{AS}) 
we have
\begin{multline*}
\frac{\pi}4\sum_{j=0}^{2m}\left[e^{\frac{2\pi ji}{2m+1}}
\tan\left(e^{\frac{2\pi ji}{2m+1}}\frac{\pi x}2\right)
-e^{\frac{2\pi ji}{2m+1}}\sec\left(e^{\frac{2\pi ji}{2m+1}}\frac{\pi x}2\right)\right]=\\
(2m+1)\left[\sum_{\text{$k\ge1$ odd}}^\infty t(\overline{(2m+1)k})x^{(2m+1)k-1}
+\sum_{\text{$k\ge 2$ even}}^\infty t((2m+1)k)x^{(2m+1)k-1}\right]
\end{multline*}
or
\[
\frac{\pi}4\sum_{j=0}^{2m}\left[\eta^j\tan\left(\eta^j\frac{\pi x}2\right)
-(-1)^j\eta^j\sec\left(\eta^j\frac{\pi x}2\right)\right]
=(2m+1)x^{2m}\hat\th_{2m+1}P(x^{2m+1})
\]
for $\eta=e^{\frac{\pi i}{2m+1}}$.  The integral of the left-hand side is
\begin{multline*}
\frac12\sum_{j=0}^{2m}\left[\log\left(\sec\left(\eta^j\frac{\pi x}2\right)\right)
-(-1)^j\log\left(\sec\left(\eta^j\frac{\pi x}2\right)+
\tan\left(\eta^j\frac{\pi x}2\right)\right)\right]\\
=\log\prod_{j=0}^{2m}\left(1+(-1)^j\sin\left(\eta^j\frac{\pi x}2\right)
\right)^{-\frac12} .
\end{multline*}
Since
\[
\frac{d}{dx}\log H(x^{2m+1})=\frac{H'(x^{2m+1})}{H(x^{2m+1})}(2m+1)x^{2m}=
(2m+1)x^{2m}P(x^{2m+1})
\]
we have
\[
\hat\th_{2m+1}H(x^{2m+1})=
\prod_{j=0}^{2m}\left(1+(-1)^j\sin\left(e^{\frac{\pi j i}{2m+1}}\frac{\pi x}2\right)
\right)^{-\frac12} ,
\]
and the conclusion follows using $E(t)=H(-t)^{-1}$.
\end{proof}
Our result for $t(\{\bar1\}_k)$ is as follows (cf. \cite[Eq. (62)]{BBB}
for $\zt(\{\bar1\}_k)$).
\begin{cor}  For all positive integers $k$,
\[
t(\{\bar1\}_k)=(-1)^{\lfloor\frac{k+1}2\rfloor}\frac{\pi^k}{2^{2k}k!} .
\]
\end{cor}
\begin{proof}
From the preceding result
\[
1+\sum_{k=1}^\infty t(\{\bar1\}_k)x^k=\sqrt{1-\sin\frac{\pi x}2}
\]
so it suffices to show that
\[
\sqrt{1-\sin z}
=\sum_{n=0}^\infty (-1)^{\lfloor\frac{n+1}2\rfloor}\frac{z^n}{2^nn!} .
\]
This can be seen by writing the left-hand side as
\begin{equation}
\label{sineq}
\sum_{n=0}^\infty\frac{(-1)^nz^{2n}}{2^{2n}(2n)!}-
\sum_{n=1}^\infty\frac{(-1)^nz^{2n-1}}{2^{2n-1}(2n-1)!}=
\cos\frac{z}2-\sin\frac{z}2
\end{equation}
and then noting that the right-hand side of Eq. (\ref{sineq}) squares 
to $1-\sin z$.
\end{proof}
We also have the following formula for $t(\{\bar3\}_k)$.
\begin{cor}  For all positive integers $k$,
\[
t(\{\bar3\}_k)=(-1)^{\lfloor\frac{k+1}2\rfloor}\frac{3\pi^{3k}}{2^{3k+1}(3k)!} .
\]
\end{cor}
\begin{proof}
In view of Theorem \ref{barrep} it suffices to show
\begin{multline*}
\sqrt{(1-\sin x)(1+\sin(\om x))(1-\sin(\om^2x))}=\\
\frac12\left[\cos x+\cos(\om x)+\cos(\om^2x)-1+\sin x-\sin(\om x)+\sin(\om^2x)
\right]
\end{multline*}
for $\om=e^{\frac{\pi i}3}$.
To show this, first express the right-hand side as
\[
2\cos\frac{x}2\cos\frac{\om x}2\cos\frac{\om^2x}2-1
+2\sin\frac{x}2\sin\frac{\om x}2\sin\frac{\om^2x}2
\]
and square it.  Now use double-angle formulas for cosine and sine
to rewrite the result in terms of sine and cosine of $x$, $\om x$, and 
$\om^2x$, and compare to the square of the left-hand side:  after
cancellation, the two can be seen to be the same using the addition
formula for cosine and the equality $1+\om^2=\om$.
\end{proof}

\newpage
\appendix
\section*{Appendix A: Multiple $t$-values of weight $\le7$}
\[
t(2,1)=-\frac12t(3)+t(2)\log2
\]
\begin{align*}
t(3,1) &= -\frac{37}{60}t(4)-\frac12\zt(\bar3,1)+t(3)\log2\\
t(2,2) &=\frac14t(4)\\
t(2,1,1) &= \frac{11}{60}t(4)+\frac14\zt(\bar 3,1)
-\frac12t(3)\log2+\frac12t(2)\log^22
\end{align*}
\begin{align*}
t(4,1) &= -\frac12t(5)-\frac17t(2)t(3)+t(4)\log2\\
t(3,2) &= -\frac12t(5)+\frac37t(2)t(3)\\
t(2,3) &= -\frac12t(5)+\frac47t(2)t(3)\\
t(3,1,1) &= -\frac{23}{248}t(5)+\frac5{21}t(2)t(3)-\frac12\zt(\bar3,1,1)
-\frac{37}{60}t(4)\log2-\frac12\zt(\bar3,1)\log2\\
&+\frac12t(3)\log^22\\
t(2,2,1) &=\frac18t(5)-\frac3{14}t(2)t(3)+\frac14t(4)\log2\\
t(2,1,2) &=\frac34t(5)-\frac12t(2)t(3)\\
t(2,1,1,1) &=-\frac{35}{248}t(5)+\frac2{21}t(2)t(3)+\frac14\zt(\bar3,1,1)
+\frac{11}{60}t(4)\log2+\frac14\zt(\bar3,1)\log2\\
&-\frac14t(3)\log^22+\frac16t(2)\log^32
\end{align*}
\begin{align*}
t(5,1) &= -\frac{73}{84}t(6)+\frac{17}{98}t(3)^2-\frac12\zt(\bar5,1)
+t(5)\log2\\
t(4,2) &= -\frac17t(6)+\frac17t(3)^2\\
t(3,3) &= -\frac12t(6)+\frac12t(3)^2\\
t(2,4) &= \frac{11}{28}t(6)-\frac17t(3)^2\\
t(4,1,1) &= \frac{45}{112}t(6)-\frac{31}{196}t(3)^2
+\frac14\zt(\bar5,1)
-\frac12t(5)\log2-\frac17t(2)t(3)\log2+\frac12t(4)\log^22\\
t(3,2,1) &=\frac{37}{112}t(6)-\frac{33}{98}t(3)^2+\frac14\zt(\bar5,1)
-\frac12t(5)\log2+\frac37t(2)t(3)\log2\\
t(3,1,2) &=\frac{61}{168}t(6)-\frac9{28}t(3)^2\\
t(2,3,1) &=-\frac2{21}t(6)-\frac3{196}t(3)^2-\frac12t(2)\zt(\bar3,1)+\frac14
\zt(\bar5,1)-\frac12t(5)\log2+\frac47t(2)t(3)\log2\\
t(2,1,3) &=\frac{27}{112}t(6)-\frac5{28}t(3)^2+\frac12t(2)\zt(\bar3,1)\\
t(2,2,2) &= \frac1{48}t(6)\\
t(3,1,1,1) &=\frac5{56}t(6)-\frac{27}{196}t(3)^2+\frac9{16}\zt(\bar5,1)
-\frac12\zt(\bar3,1,1,1)-\frac{23}{248}t(5)\log2\\
&+\frac5{21}t(2)t(3)\log2-\frac12\zt(\bar3,1,1)\log2
-\frac{37}{120}t(4)\log^22-\frac14\zt(\bar3,1)\log^22\\
&+\frac16t(3)\log^32\\
t(2,2,1,1) &=\frac{23}{336}t(6)-\frac{2}{49}t(3)^2+\frac14t(2)\zt(\bar3,1)
-\frac1{16}\zt(\bar5,1)+\frac18t(5)\log2-\frac3{14}t(2)t(3)\log2\\
&+\frac18t(4)\log^22\\
t(2,1,2,1) &
=-\frac{11}{28}t(6)+\frac{121}{392}t(3)^2-\frac38\zt(\bar5,1)
+\frac34t(5)\log2-\frac12t(2)t(3)\log2\\
t(2,1,1,2) &=-\frac1{16}t(6)+\frac18t(3)^2-\frac14t(2)\zt(\bar3,1)\\
t(2,1,1,1,1) &=-\frac3{112}t(6)+\frac2{49}t(3)^2-\frac3{16}\zt(\bar5,1)
+\frac14\zt(\bar3,1,1,1)-\frac{35}{248}t(5)\log2\\
&+\frac2{21}t(2)t(3)\log2+\frac14\zt(\bar3,1,1)\log2
+\frac{11}{120}t(4)\log^22+\frac18\zt(\bar3,1)\log^22\\
&-\frac1{12}t(3)\log^32+\frac1{24}t(2)\log^42
\end{align*}
\begin{align*}
t(6,1) &= -\frac12t(7)-\frac17t(3)t(4)-\frac1{31}t(2)t(5)+t(6)\log2\\
t(5,2) &= -\frac12t(7)+\frac27t(3)t(4)+\frac5{31}t(2)t(5)\\
t(4,3) &= -\frac12t(7)+\frac67t(3)t(4)-\frac{10}{31}t(2)t(5)\\
t(3,4) &= -\frac12t(7)+\frac17t(3)t(4)+\frac{10}{31}t(2)t(5)\\
t(2,5) &= -\frac12t(7)-\frac27t(3)t(4)+\frac{26}{31}t(2)t(5)\\
t(5,1,1) &=\frac{682}{2159}t(7)-\frac{429}{2380}t(3)t(4)+\frac{146}{1581}
t(2)t(5)+\frac2{17}\zt(\bar5,1,1)-\frac7{68}\zt(\bar3,3,1)\\
&-\frac{73}{84}t(6)\log2+\frac{17}{98}t(3)^2\log2-\frac12\zt(\bar5,1)\log2
+\frac12t(5)\log^22\\
t(4,2,1) &= \frac{19887}{34544}t(7)-\frac{1033}{1785}t(3)t(4)
+\frac{27}{1054}t(2)t(5)+\frac3{17}\zt(\bar5,1,1)-\frac1{34}\zt(\bar3,3,1)\\
&-\frac17t(6)\log2+\frac17t(3)^2\log2\\
t(4,1,2) &= \frac58t(7)-\frac{19}{28}t(3)t(4)+\frac5{62}t(2)t(5)\\
t(3,3,1) &=-\frac{20229}{17272}t(7)+\frac{56}{85}t(3)t(4)
+\frac{407}{1054}t(2)t(5)-\frac12t(3)\zt(\bar3,1)-\frac{21}{34}\zt(\bar5,1,1)\\
&+\frac7{68}\zt(\bar3,3,1)-\frac12t(6)\log2+\frac12t(3)^2\log2\\
t(3,1,3) & = \frac{28865}{8636}t(7)-\frac{1973}{1020}t(3)t(4)-\frac{560}{527}
t(2)t(5)+\frac12t(3)\zt(\bar3,1)+\frac{21}{17}\zt(\bar5,1,1)\\
&-\frac{7}{34}\zt(\bar3,3,1)\\
t(3,2,2) &= \frac3{16}t(7)+\frac{3}{28}t(3)t(4)-\frac{15}{62}t(2)t(5)\\
t(2,3,2) &= \frac58t(7)-\frac12t(2)t(5)\\
t(2,2,3) &= \frac3{16}t(7)+\frac17t(3)t(4)-\frac8{31}t(2)t(5)\\
t(2,4,1) &= -\frac{6933}{34544}t(7)+\frac{2347}{7140}t(3)t(4)
-\frac{265}{1054}t(2)t(5)-\frac3{17}\zt(\bar5,1,1)+\frac1{34}\zt(\bar3,3,1)\\
&+\frac{11}{28}t(6)\log2-\frac17t(3)^2\log2\\
t(2,1,4) &= \frac58t(7)+\frac5{28}t(3)t(4)-\frac{18}{31}t(2)t(5)\\
\end{align*}
\begin{align*}
t(4,1,1,1)&=-\frac{35117}{69088}t(7)+\frac{1801}{3570}t(3)t(4)
-\frac{106}{1581}t(2)t(5)-\frac5{34}\zt(\bar5,1,1)
+\frac9{136}\zt(\bar3,3,1)\\
&+\frac{45}{112}t(6)\log2-\frac{31}{196}t(3)^2\log2+
\frac14\zt(\bar5,1)\log2-\frac14t(5)\log^22\\
&-\frac1{14}t(2)t(3)\log^22+\frac16t(4)\log^32\\
t(3,2,1,1)&=\frac{2373}{4064}t(7)-\frac{51}{140}t(3)t(4)-\frac16t(2)t(5)
+\frac14t(3)\zt(\bar3,1)+\frac14\zt(\bar5,1,1)+\frac{37}{112}t(6)\log2\\
&-\frac{33}{98}t(3)^2\log2+\frac14\zt(\bar5,1)\log2
-\frac14t(5)\log^22+\frac3{14}t(2)t(3)\log^22\\
t(3,1,2,1)&=-\frac{77617}{69088}t(7)+\frac{3903}{4760}t(3)t(4)
+\frac{207}{1054}t(2)t(5)-\frac{27}{68}\zt(\bar5,1,1)
+\frac9{136}\zt(\bar3,3,1)\\
&+\frac{61}{168}t(6)\log2-\frac9{28}t(3)^2\log2\\
t(3,1,1,2)&=-\frac{126255}{69088}t(7)+\frac{13301}{14280}t(3)t(4)
+\frac{365}{527}t(2)t(5)-\frac14t(3)\zt(\bar3,1)-\frac{21}{34}\zt(\bar5,1,1)\\
&+\frac7{68}\zt(\bar3,3,1)\\
t(2,3,1,1)&=-\frac{8297}{69088}t(7)+\frac{152}{357}t(3)t(4)
-\frac{2921}{12648}t(2)t(5)-\frac12t(2)\zt(\bar3,1,1)+\frac1{34}\zt(\bar5,1,1)\\
&+\frac5{136}\zt(\bar3,3,1)
-\frac2{21}t(6)\log2-\frac3{196}t(3)^2\log2-\frac12t(2)\zt(\bar3,1)\log2
+\frac14\zt(\bar5,1)\log2\\
&-\frac14t(5)\log^22+\frac27t(2)t(3)\log^22\\
t(2,1,3,1)&=\frac{43073}{69088}t(7)-\frac{11813}{7140}t(3)t(4)
+\frac{1813}{2108}t(2)t(5)+t(2)\zt(\bar3,1,1)+\frac14t(3)\zt(\bar3,1)\\
&+\frac{27}{68}\zt(\bar5,1,1)-\frac9{136}\zt(\bar3,3,1)
+\frac{27}{112}t(6)\log2-\frac5{28}t(3)^2\log2+\frac12t(2)\zt(\bar3,1)\log2\\
t(2,1,1,3)&=-\frac{126255}{69088}t(7)+\frac{2587}{1785}t(3)t(4)
+\frac{1169}{4216}t(2)t(5)-\frac12t(2)\zt(\bar3,1,1)-\frac14t(3)\zt(\bar3,1)\\
&-\frac{21}{34}\zt(\bar5,1,1)+\frac7{68}\zt(\bar3,3,1)\\
t(2,2,2,1)&=-\frac1{32}t(7)-\frac3{56}t(3)t(4)+\frac{15}{248}t(2)t(5)+\frac1{48}t(6)\log2\\
t(2,2,1,2)&=-\frac{15}{32}t(7)-\frac3{56}t(3)t(4)+\frac{53}{124}t(2)t(5)\\
t(2,1,2,2)&=-\frac{15}{32}t(7)-\frac1{14}t(3)t(4)+\frac{71}{248}t(2)t(5)\\
\end{align*}
\begin{align*}
t(3,1,1,1,1)&=\frac{15729}{17272}t(7)-\frac{1259}{2380}t(3)t(4)
-\frac{431}{1581}t(2)t(5)+\frac{93}{272}\zt(\bar5,1,1)
+\frac5{136}\zt(\bar3,3,1)\\
&-\frac12\zt(\bar3,1,1,1,1)+\frac5{56}t(6)\log2-\frac{27}{196}t(3)^2\log2
+\frac9{16}\zt(\bar5,1)\log2\\
&-\frac12\zt(\bar3,1,1,1)\log2-\frac{23}{496}t(5)\log^22
+\frac5{42}t(2)t(3)\log^22-\frac14\zt(\bar3,1,1)\log^22\\
&-\frac{37}{360}t(4)\log^32-\frac1{12}\zt(\bar3,1)\log^32
+\frac1{24}t(3)\log^42\\
t(2,2,1,1,1)&=\frac5{2032}t(7)-\frac2{21}t(3)t(4)+\frac{55}{744}t(2)t(5)
+\frac14t(2)\zt(\bar3,1,1)-\frac1{16}\zt(\bar5,1,1)\\
&+\frac{23}{336}t(6)\log2-\frac{2}{49}t(3)^2\log2
+\frac14t(2)\zt(\bar3,1)\log2-\frac1{16}\zt(\bar5,1)\log2\\
&+\frac1{16}t(5)\log^22
-\frac3{28}t(2)t(3)\log^22+\frac1{24}t(4)\log^32\\
t(2,1,2,1,1)&=\frac{341}{2032}t(7)+\frac27t(3)t(4)-\frac{89}{248}t(2)t(5)
-\frac14t(2)\zt(\bar3,1,1)-\frac18t(3)\zt(\bar3,1)\\
&-\frac1{16}\zt(\bar3,3,1)
-\frac{11}{28}t(6)\log2+\frac{121}{392}t(3)^2\log2
-\frac38\zt(\bar5,1)\log2\\
&+\frac38t(5)\log^22-\frac14t(2)t(3)\log^22\\
t(2,1,1,2,1)&=\frac{13353}{34544}t(7)+\frac{419}{14280}t(3)t(4)
-\frac{1529}{4216}t(2)t(5)-\frac14t(2)\zt(\bar3,1,1)
+\frac{21}{136}\zt(\bar5,1,1)\\
&-\frac7{272}\zt(\bar3,3,1)
-\frac1{16}t(6)\log2+\frac18t(3)^2\log2
-\frac14t(2)\zt(\bar3,1)\log2\\
t(2,1,1,1,2)&=\frac{21989}{17272}t(7)-\frac{10093}{14280}t(3)t(4)
-\frac{1715}{4216}t(2)t(5)+\frac14t(2)\zt(\bar3,1,1)+\frac18t(3)\zt(\bar3,1)\\
&+\frac{21}{68}\zt(\bar5,1,1)-\frac7{136}\zt(\bar3,3,1)\\
t(2,1,1,1,1,1)&=-\frac{2407}{4064}t(7)+\frac7{30}t(3)t(4)
+\frac{26}{93}t(2)t(5)
-\frac{3}{16}\zt(\bar5,1,1)+\frac14\zt(\bar3,1,1,1,1)\\
&-\frac3{112}t(6)\log2+\frac2{49}t(3)^2\log2-\frac3{16}\zt(\bar5,1)\log2
+\frac14\zt(\bar3,1,1,1)\log2\\
&-\frac{35}{496}t(5)\log^22
+\frac1{21}t(2)t(3)\log^22+\frac18\zt(\bar3,1,1)\log^22\\
&+\frac{11}{360}t(4)\log^32+\frac1{24}\zt(\bar3,1)\log^32
-\frac1{48}t(3)\log^42+\frac1{120}t(2)\log^52
\end{align*}
\section*{Appendix B: Multiple $t$-values in terms of Saha elements}
\par\noindent
\[
t(4)=4t(2,2)
\]
\begin{align*}
t(5)&=7t(3,2)+6t(2,1,2)\\
t(4,1)&=\frac12t(3,2)-t(2,1,2)+4t(2,2,1)\\
t(2,3)&=\frac52t(3,2)+t(2,1,2)
\end{align*}
\begin{align*}
t(6)&=48t(2,2,2)\\
t(5,1)&=-\frac{19}9t(3,1,2)-\frac{25}9t(2,2,2)+7t(3,2,1)+6t(2,1,2,1)\\
t(4,2)&=-\frac49t(3,1,2)+\frac89t(2,2,2)\\
t(3,3)&=-\frac{14}9t(3,1,2)+\frac{28}9t(2,2,2)\\
t(2,4)&=\frac49t(3,1,2)+\frac{100}9t(2,2,2)\\
t(4,1,1)&=\frac1{18}t(3,1,2)-\frac{173}{18}t(2,2,2)+4t(2,1,1,2)
+\frac12t(3,2,1)-t(2,1,2,1)+4t(2,2,1,1)\\
t(2,3,1)&=-\frac56t(3,1,2)-\frac{29}6t(2,2,2)+2t(2,1,1,2)+\frac52t(3,2,1)
+t(2,1,2,1)\\
t(2,1,3)&=-\frac29t(3,1,2)+\frac{85}9t(2,2,2)-2t(2,1,1,2)
\end{align*}
\begin{align*}
t(7)&=-\frac{211}{87}t(3,2,2)+\frac{62}{29}t(2,2,1,2)-\frac{152}{29}t(2,1,2,2)\\
t(6,1)&=\frac{4723}{174}t(3,2,2)+\frac{233}{29}t(2,2,1,2)
+\frac{20}{29}t(2,1,2,2)+48t(2,2,2,1)\\
t(5,2)&=\frac{1397}{348}t(3,2,2)+\frac{151}{58}t(2,2,1,2)+\frac2{29}t(2,1,2,2)\\
t(4,3)&=\frac{1861}{174}t(3,2,2)+\frac{151}{29}t(2,2,1,2)+\frac4{29}t(2,1,2,2)\\
t(3,4)&=\frac{797}{348}t(3,2,2)+\frac{127}{58}t(2,2,1,2)-\frac6{29}t(2,1,2,2)\\
t(2,5)&=-\frac{509}{174}t(3,3,2)+\frac{33}{29}t(2,2,1,2)-\frac{36}{29}t(2,1,2,2)\\
t(5,1,1)&=\frac{10405}{4176}t(3,2,2)+\frac{3119}{696}t(2,2,1,2)
+\frac{167}{174}t(2,1,2,2)+7t(3,1,1,2)\\
&+6t(2,1,1,1,2)-\frac{19}9t(3,1,2,1)
-\frac{25}9t(2,2,2,1)+7t(3,2,1,1)+6t(2,1,2,1,1)\\
t(4,2,1)&=-\frac{8735}{4176}t(3,2,2)-\frac{709}{696}t(2,2,1,2)
-\frac{7}{174}t(2,1,2,2)-\frac49t(3,1,2,1)+\frac89t(2,2,2,1)\\
t(4,1,2)&=-\frac{12235}{1392}t(3,2,2)-\frac{1081}{232}t(2,2,1,2)
-\frac{11}{58}t(2,1,2,2)\\
t(3,3,1)&=\frac{1745}{1044}t(3,2,2)-\frac{197}{174}t(2,2,1,2)
+\frac2{87}t(2,1,2,2)+2t(3,1,1,2)-\frac{14}9t(3,1,2,1)\\
&+\frac{28}9t(2,2,2,1)\\
t(3,1,3)&=-\frac{595}{1392}t(3,2,2)+\frac{127}{232}t(2,2,1,2)
-\frac3{58}t(2,1,2,2)-2t(3,1,1,2)\\
t(2,3,2)&=-\frac{295}{348}t(3,2,2)-\frac{93}{58}t(2,2,1,2)-\frac2{29}t(2,1,2,2)\\
t(2,2,3)&=\frac{649}{464}t(3,2,2)+\frac{57}{232}t(2,2,1,2)-\frac5{58}t(2,1,2,2)\\
t(2,4,1)&=\frac{11887}{2088}t(3,2,2)+\frac{269}{348}t(2,2,1,2)
+\frac{11}{87}t(2,1,2,2)+\frac49t(3,1,2,1)+\frac{100}9t(2,2,2,1)\\
t(2,1,4)&=\frac{55}{48}t(3,2,2)-\frac38t(2,2,1,2)-\frac12t(2,1,2,2)
\end{align*}
\begin{align*}
t(4,1,1,1)&=-\frac{11503}{2088}t(3,2,2)-\frac{539}{348}t(2,2,1,2)
-\frac{25}{174}t(2,1,2,2)+\frac12t(3,1,1,2)\\
&-t(2,1,1,1,2)+\frac1{18}t(3,1,2,1)-\frac{173}{18}t(2,2,2,1)+4t(2,1,1,2,1)\\
&+\frac12t(3,2,1,1)-t(2,1,2,1,1)+4t(2,2,1,1,1)\\
t(2,3,1,1)&=-\frac{1355}{696}t(3,2,2)-\frac{31}{116}t(2,2,1,2)
+\frac{9}{58}t(2,1,2,2)+\frac52t(3,1,1,2)+t(2,1,1,1,2)\\
&-\frac56t(3,1,2,1)-\frac{29}6t(2,2,2,1)
+2t(2,1,1,2,1)+\frac52t(3,2,1,1)+t(2,1,2,1,1)\\
t(2,1,3,1)&=\frac{7787}{1044}t(3,2,2)+\frac{859}{174}t(2,2,1,2)
+\frac{35}{87}t(2,1,2,2)+2t(2,1,1,1,2)-\frac29t(3,1,2,1)\\
&+\frac{85}9t(2,2,2,1)-2t(2,1,1,2,1)\\
t(2,1,1,3)&=-\frac{113}{174}t(3,2,2)-\frac{37}{29}t(2,2,1,2)
-\frac{15}{29}t(2,1,2,2)-2t(2,1,1,1,2)
\end{align*}

\begin{thebibliography}{99}
\bibitem{ABS} M. Aguiar, N. Bergeron, and F. Sotille, Combinatorial Hopf
algebras and generalized Dehn-Somerville relations, \emph{Compos. Math.}
{\bf 142} (2006), 1-30.
\bibitem{AS}
M. Abramowitz and I. E. Stegun, Handbook of Mathematical Functions,
U. S. Government Printing Office, Washington, DC, 1964; reprinted by
Dover, New York, 1972.
\bibitem{BJOP}
M. Bigotte, G. Jacob, N. E. Oussous, and M. Petitot,
Lyndon words and shuffle algebras for generating the coloured multiple
zeta values relations tables, \emph{Theoret. Comput. Sci.} {\bf 273} (2002),
271-282.
\bibitem{BBV}
J. Bl\"umlein, D. J. Broadhurst, and J. A. M. Vermaseren,
The multiple zeta value data mine, \emph{Comput. Phys. Commun.} {\bf 181}
(2010), 582-625.
\bibitem{BBB}
D. J. Broadhurst, J. M. Borwein, and D. M. Bradley,
Evaluation of $k$-fold Euler/Zagier sums:  a compendium of results for
arbitrary $k$, \emph{Electron. J. Combin} {\bf 4(2)} (1997), art. 5 (21 pp).
\bibitem{Br}
F. Brown, Mixed Tate Modules over {\bf Z}, \emph{Ann. of Math.} {\bf 175}
(2012), 949-976.
\bibitem{Che}
H. Chen, Evaluation of some variant Euler sums, \emph{J. Integer Sequences}
{\bf 9} (2006), art. 06.2.3 (9 pp).
\bibitem{Chung}
C-L. Chung, On the sum relation of multiple Hurwitz zeta functions,
\emph{Quaest. Math.} (to appear); preprint {\tt arXiv 1609.01362[NT]}.
\bibitem{E}
L. Euler, Meditationes circa singulare serierum genus, \emph{Novi
Comm. Acad. Sci. Petropol.} {\bf 20} (1776), 140-185; reprinted in
\emph{Opera Omnia}, ser. I, vol. 15, B. G. Teubner, Berlin, 1927,
pp. 217-267.
\bibitem{GR}
I. S. Gradshteyn and I. M. Ryzhik,\emph{Table of Integrals, Series and 
Products} (corrected and enlarged by A. Jeffrey), Academic Press, 
New York, 1980.
\bibitem{H1992}
M. E. Hoffman, Multiple harmonic series, \emph{Pacific J. Math.}
{\bf 152} (1992), 275-290.
\bibitem{H1997}
M. E. Hoffman, The algebra of multiple harmonic series, \emph{J. Algebra}
{\bf 194} (1997), 477-495.
\bibitem{H2000}
M. E. Hoffman, Quasi-shuffle products, \emph{J. Algebraic Combin.}
{\bf 11} (2000), 49-68.
\bibitem{H2009} 
M. E. Hoffman, Rooted trees and symmetric functions:  Zhao's homomorphism
and the commutative hexagon, in \emph{Vertex Operator Algebras and 
Related Areas} (Contemp. Math. vol. 497), M. Bergvelt {\it et. al.}
(eds.), American Mathematical Society, Providence, RI, 2009, pp. 95-95.
\bibitem{H2015}
M. E. Hoffman, Quasi-symmetric functions and mod $p$ multiple harmonic sums,
\emph{Kyushu J. Math.} {\bf 69} (2015), 345-366.
\bibitem{IKZ}
K. Ihara, M. Kaneko, and D. Zagier, Derivation and double shuffle relations
for multiple zeta values, \emph{Compositio Math.} {\bf 142} (2006), 307-338.
\bibitem{KT}
M. Kaneko and K. Tasaka, Double zeta values, double Eisenstein series,
and modular forms of level 2, \emph{Math. Ann.} {\bf 307} (2013), 1091-1118.
\bibitem{KTs}
M. Kaneko and H. Tsumura, On a variant of multiple zeta values of level
two, preprint.
\bibitem{KK}
E. D. Krupnikov and K. S. K\"olbig, Some special cases of the generalized
hypergeometric function ${}_{q+1}F_q$, \emph{J. Comput. Appl. Math.} {\bf 78}
(1997), 79-95.
\bibitem{Mac}
I. G. Macdonald, \emph{Symmetric Functions and Hall Polynomials}, 2nd ed.,
Oxford University Press, New York, 1995.
\bibitem{M}
S. Muneta, On some explicit evaluations of multiple zeta-star values,
\emph{J. Number Theory} {\bf 128} (2008), 2538-2548.
\bibitem{MS}
M. Ram Murty and K. Sinha, Multiple Hurwitz zeta functions, in
\emph{Multiple Dirichlet Series, Automorphic Forms and Analytic Number
Theory}, S. Friedberg (ed.), American Mathematical Society, Providence,
2006, pp. 135-156.
\bibitem{NT}
T. Nakamura and K. Tasaka, Remarks on double zeta values of level 2,
\emph{J. Number Theory} {\bf 133} (2013), 48-54.
\bibitem{N}
N. Nielsen, \emph{Handbuch der Theorie der Gammafunktion}, B. G. Teubner,
Leipzig, 1906; reprinted in \emph{Die Gammafunktion}, Chelsea, New York,
1965.
\bibitem{O}
F. W. J. Olver {\it et. al.}, \emph{NIST Digital Library of Mathematical
Functions}, {\tt http://dlmf.nist.gov}, release 1.0.19 of June 22, 2018.
\bibitem{Sa}
B. Saha, A conjecture about multiple $t$-values, preprint 
{\tt arXiv: 1712.06325[NT]}.
\bibitem{SJ2017}
Z. Shen and L. Jia, Some identities for multiple Hurwitz zeta values,
\emph{J. Number Theory} {\bf 179} (2017), 256-267.
\bibitem{Z}
J. Zhao, \emph{Multiple zeta functions, multiple polylogarithms, and their
special values}, World Scientific, Singapore, 2016.
\end{thebibliography}
\end{document}